\documentclass[12pt,A4paper]{article}
\usepackage{graphicx} 
\usepackage[utf8]{inputenc}
\usepackage{amsthm,amsmath,amssymb,graphicx,enumerate,comment}
\usepackage{xcolor,colortbl}
\usepackage{pdfcomment} 

\newtheorem{theorem}{Theorem}[section]
\newtheorem{lemma}[theorem]{Lemma}
\newtheorem{proposition}[theorem]{Proposition}
\newtheorem{corollary}[theorem]{Corollary}

\newtheorem{conjecture}{Conjecture}
\theoremstyle{definition}
\newtheorem{definition}[theorem]{Definition}

\newtheorem{remark}[theorem]{Remark}

\newcommand{\ZZ}{\mathbb{Z}}
\newcommand{\B}{{\cal B}}

\newcommand{\arxiv}[2]{\href{https://arxiv.org/abs/#1}{\texttt{arXiv:#1}} \texttt{[#2]}}

\title{On the hamiltonicity problem of bicirculants: \\
a reduction to cyclic Haar graphs}
\author{Simona\,Bonvicini\thanks{Dipartimento di Scienze Fisiche, Informatiche e Matematiche, 
Universit\`a di Modena e Reggio Emilia, via Campi 213/b, 41126 Modena, Italy.},
\ Toma\v z\, Pisanski\thanks{University of Primorska  and Institute of Mathematics, Physics and Mechanics, Slovenia},
\ Arjana\,\v Zitnik\thanks{University of Ljubljana and Institute of Mathematics, Physics and Mechanics, Ljubljana, Slovenia}}
\date{\today}

\begin{document}

\maketitle

\begin{abstract}

A \emph{bicirculant} is a regular graph that admits an automorphism having two vertex-orbits of the same size. 
A  bicirculant can be described as follows. Given an integer $m \ge 1$ and sets $R, S, T \subseteq \ZZ_m$ such that $R=-R$, $T=-T$, 
$0 \not\in R \cup T$ and $0 \in S$, the graph $B(m;R,S,T)$ has vertex set $V=\{u_0,\dots,u_{m-1},v_0,\dots,v_m-1\}$ and edge set 
$E=\{u_iu_{i+j}| \ i \in\ZZ_m, j \in R\} \cup 
      \{v_iv_{i+j}| \ i \in\ZZ_m, j \in T\} \cup
      \{u_iv_{i+j}| \ i \in\ZZ_m, j \in S\}.$
Bicirculant graphs with $R=T=\emptyset$ are known as \emph{cyclic Haar graphs}.

In 2025 we conjectured that the only non-hamiltonian graphs among regular connected bicirculants of degree more than one are the generalized Petersen graphs $G(m,2)$ with $m \equiv 5 \pmod 6$. Recently we have verified the conjecture for bicirculants with $|S|\le 2$ 
and for bicirculants with $|R|=|T|$ odd. 

In this paper we show that the conjecture holds for all bicirculants with 
$|S| \le 3$ and for all bicirculants with 
$|S| \ge 4$ and $m/\gcd(m, S)$ even. 
As a byproduct of our results, we prove that every connected bicirculant graph on $2m$ vertices with 
$|S| \ge 4$ is hamiltonian 
for even  $m< 9\, 240$,  and for odd $m< 3\,465$.
Finally, we show that the existence of a hamilton cycle in every connected cyclic Haar graph of valence at least $4$ 
implies that every connected bicirculant graph of valence 
at least $4$ is hamiltonian. 
\end{abstract}

\textbf{Keywords:} bicirculant graph, hamilton cycle,  
cyclic Haar graph.

\textbf{Math Subj Class (2020)}: 
05C45, 
05C25, 
05C76,  
05C70,  
05E18. 

\section{Motivation and main results}

A \emph{bicirculant} is a regular graph that is a cyclic cover over a two-vertex (pre)-graph. In an equivalent way, bicirculants are defined to be regular graphs that admit an automorphism having two vertex-orbits of the same size \cite{Pi2007}. 

Bicirculants are a special class of polycirculants  or  $k$-circulants.  A $k$-\emph{circulant} is a (regular) graph possessing an automorphism with $k$ orbits of the same size.
For $k = 1$ we get the known circulant graphs, for $k=2$ bicirculants, for $k = 3$ tricirculants, etc. 

Basic motivation for the study of polycirculants comes from an old problem of Maru\v{s}i\v{c} \cite{Ma1981}, asking whether every vertex-transitive graph is a polycirculant. The problem is still not completely solved. Another motivation comes from the study of polycyclic configurations, first introduced in \cite{BoPi2003}. Namely, the Levi graph of a polycyclic configuration is a polycirculant. For more on configurations, see \cite{Gr2009,PiSe2013}. Recently, the fact that the smallest semi-symmetric graph is a polycirculant was used to find its unit-distance realization in the plane \cite{BeGePi2023a}. Polycirculants and bicirculants in particular, admit several generalizations. One of them, where the cyclic cover is replaced by a more general regular cover, has been studied, for instance in \cite{CoEsPi2018}. Others are motivated by generalizing some well-known families of bicirculants, like the generalized Petersen graphs \cite{JeJaPi2020}.

A typical bicirculant can be described as follows; see for example \cite{MALNIC2007891}. Given an integer $m \ge 1$ and sets $R,S,T \subseteq \ZZ_m$
such that $R=-R$, $T=-T$, 
$0 \not\in R \cup T$   and $0 \in S$, 
the graph $B(m;R,S,T)$ has vertex set 
$V=V_1 \cup V_2$, where $V_1=\{u_0,\dots,u_{m-1}\}$ and $V_2=\{v_0,\dots,v_{m-1}\}$, and edge set 
$$E=\{u_iu_{i+j}| \ i \in\ZZ_m, j \in R\} \cup 
      \{v_iv_{i+j}| \ i \in\ZZ_m, j \in T\} \cup
      \{u_iv_{i+j}| \ i \in\ZZ_m, j \in S\}.$$ 
Obviously, the mapping $\alpha:V \to V$, defined by $\alpha(u_i)=u_{i+1}$, $\alpha(v_i)=v_{i+1}$
is an automorphism of $B(m;R,S,T)$, having two vertex-orbits of the same size.

We call the vertices from $V_1$ the \emph{outer vertices} and the vertices from $V_2$ the \emph{inner vertices}.
There are three types of edges: the edges  incident  to two outer vertices are called \emph{outer edges},
the edges   incident to two inner vertices are called \emph{inner edges}, and edges connecting an outer vertex to an inner vertex are called  \emph{spokes}. Specifically, the edges $u_iu_{i+a}$, $i \in \ZZ_m$, $a \in R$, are called \emph{outer edges of type} $a$, the edges $v_iv_{i+b}$, $i \in \ZZ_m$, $b \in T$, are called \emph{inner edges of type} $b$ and the edges $u_iv_{i+c}$, $i \in \ZZ_m$, $c \in S$, are called \emph{spokes of type} $c$. 
The subgraphs of $B(m;R,S,T)$ induced by the outer and inner  vertices are circulant graphs with generating sets $R$  and $T$, respectively. 
The bipartite subgraph of $B(m;R,S,T)$, spanned by the spokes, is  a \emph{cyclic Haar graph}, which we denote by $H(m;S)$; so $H(m;S)= B(m;\emptyset,S,\emptyset)$. 
The family of cyclic Haar graphs was introduced by Hladnik, Maru\v si\v c and Pisanski in  2002 \cite{HlMaPi2002}.

Many well-known graphs are bicirculants. For example, the generalized Petersen graphs are bicirculants, of valence 3. More generally, for $m \ge 3$ and $0 < j,k < m/2$, an $I$-graph $I(m;j,k)$ is a bicirculant $B(m;\{j,-j\}, \{0\},\{k,-k\}$ \cite{BoPi2017,Foster} .
A generalized rose window graph   is a bicirculant graph of degree 4, obtained from an $I$-graph by adding an additional set of spokes  \cite{wilsonRW, BoPiZi2025}.
Similarly, the  pentavalent generalized  Taba\v cjn graphs  \cite{Taba1,Taba2} are obtained from the generalized rose window graphs by adding an additional set of spokes.
Also Cayley graphs of dihedral groups are bicirculant graphs $B(m;R,S,T)$ with $R=T$.
The structured yet diverse nature of bicirculant graphs provides a particularly suitable setting for the study of hamilton cycles, which is the focus of this paper.

The existence of hamilton cycles was already considered in the past for some special classes of bicirculant graphs.
In 1983, Brian Alspach proved that the only non-hamiltonian generalized Petersen graphs are $G(m,2)$ where $m \equiv 5 \pmod 6$ \cite{Al11983}. In \cite{BoPiZi2025_2ndpaper} we named these graphs the \emph{Alspach generalized Petersen graphs}. The classification of Brian Alspach of hamiltonian generalized Petersen graphs has been extended by Bonvicini and Pisanski to $I$-graphs and  no new non-hamiltonian examples have been found \cite{BoPi2017}. In a recent paper \cite{BoPiZi2025} we proved that all generalized rose window graphs are hamiltonian and we proved that many more bicirculant graphs are hamiltonian in \cite{BoPiZi2025_2ndpaper}. Based on these results, we believe that the following conjecture (already posed in \cite{BoPiZi2025}) is true.

\begin{conjecture} \label{conject}
Every connected bicirculant, except for $K_2$ and the generalized Petersen graphs $G(m, 2)$ with $m\equiv 5\pmod 6$,  is hamiltonian. 
\end{conjecture}

The aim of this paper is to contribute further to the solution Conjecture \ref{conject}.
If true, this would imply a contribution  towards resolving Lov\'asz conjecture that every connected vertex transitive graph, except for the five known exceptions, has a hamilton cycle. Namely, it would imply that that all vertex-transitive  bicirculants except for the complete graph $K_2$ and the Petersen graph are hamiltonian. 

\medskip

In this paper we give a partial solution to Conjecture \ref{conject}.
Our approach is based on a new notion of uniform and non-uniform representations of bicirculant graphs that we introduce.
The following are our main results.

\begin{theorem}\label{th_bicirculant_s3}
A connected bicirculant graph whose vertices are incident to $s\le 3$ spokes is hamiltonian, except for the complete graph
$K_2$ and the Alspach generalized Petersen graphs.
\end{theorem}

\begin{theorem}\label{cor_main}
Let   $G=B(m; R, S, T)$ be a connected bicirculant graph with $|S|\ge 4$ and $m/\gcd(m, S)$ even. Then $G$  is hamiltonian.
\end{theorem}

\begin{theorem}\label{th_cyclich_Haar_then_bicirculant_ham}
If every connected cyclic Haar graph of valence at least $4$ is hamiltonian, then every connected bicirculant graph $B(m;R,S,T)$ with $|S|\ge 4$ is hamiltonian.
\end{theorem}

As a consequence of Theorem \ref{th_bicirculant_s3}, the following results also hold. 

\begin{corollary}\label{cor_gcdmRT=1}
Every generalized Tabačjn graph of order $2m$, $m\ge 3$,  is hamiltonian. 
\end{corollary}

The last result, Theorem \ref{thm:primes_3rd_paper},  improves the bounds from \cite[Theorem 1.2]{BoPiZi2025_2ndpaper}, where we have shown that every bicirculant of order $2m$ is hamiltonian for  even  $m<210$ and odd $m< 1\,155$, with the exception of the complete graph $K_2$ and  the generalized Petersen graphs $G(m, 2)$ with $m\equiv 5\pmod 6$. 

\begin{theorem}\label{thm:primes_3rd_paper}
Let $G=B(m;R,S,T)$ be a connected bicirculant  with  $|S| \ge 4$. The following  hold:
\begin{description}
\item[(i)]  If $m$ is even and a product of at most four prime powers, then $G$ is hamiltonian.

\item[(ii)]  If $m=2^{\ell}\,\Pi^4_{i=1}p_i$, where $\ell\in\{1, 2\}$ and each $p_i$ is an odd prime, then $G$ is hamiltonian.

\item[(iii)]  If  $m$ is odd and a product of at most three prime powers, then $G$ is hamiltonian.

\item[(iv)]  If $m$ is odd and a product of at most four prime powers and $\gcd(m,S)>1$, 
then $G$ is hamiltonian. 

\item[(v)]  If $m$ is odd, $m=\Pi^4_{i=1}p_i$, where each $p_i$ is an odd prime, then $G$ is hamiltonian. 
\end{description}

Consequently, every connected bicirculant graph on $2m$ vertices whose vertices are  incident  
to at least $4$ spokes is hamiltonian for even   $m< 9\, 240$, 
and for odd $m< 3\,465$ or $m< 15\, 015$ according to whether the spokes induce a connected subgraph or not, respectively.
\end{theorem}

\section{Background}\label{sec_preliminary}

In this section, we  report the results on the existence of hamilton cycles in bicirculant graphs that are known from the literature and will be used later in the paper.  We also review basic properties of bicirculants that will be needed.

\smallskip

Alspach and Zhang proved in \cite{AlZh1989} that  all connected Cayley graphs of valence 3 on dihedral groups are hamiltonian.  
In the subsequent paper \cite{AlChDe2010} by Alspach, Chen and Dean, a stronger property is given for Cayley graphs of valence at least three on generalized dihedral groups.
We note that  bicirculants $B(m;R,S,T)$ with $R=T$ are Cayley graphs on dihedral groups. In particular, cyclic Haar graphs which are bicirculants of form $B(m; \emptyset, S,\emptyset)$ are also Cayley graphs on dihedral groups.
The following results thus follow directly from  \cite[Theorem 3.1]{AlZh1989}, \cite[Theorem 1.8]{AlChDe2010} and \cite[Lemma 2.9] {AlChDe2010}. 

\begin{theorem}[\protect{\cite{AlZh1989}}] \label{th_cyclic_Haar_s3}
Every connected cubic cyclic Haar graph   is  hamiltonian.
\end{theorem}

\begin{theorem}[\protect{\cite{AlChDe2010}}]\label{th_generalized_dihedral}
Let $G=B(m;R,S,R)$ be a connected bicirculant graph with $m$ even and valence at least $3$. Then $G$ is hamilton-connected, unless it is bipartite, in which case it is hamilton-laceable. In particular, $G$ is hamiltonian.
\end{theorem}

\begin{lemma}[\protect{\cite{AlChDe2010}}]\label{lem_ham_path_cyclic_Haar}
Every connected  cyclic Haar graph has a hamilton path.
\end{lemma}

In the recent paper \cite{BoPiZi2025_2ndpaper} we proved the existence of a hamilton cycle for certain larger families of bicirculant graphs, which are specified in the following three theorems.

\begin{theorem}[\protect{\cite[Theorem 1.1]{BoPiZi2025_2ndpaper}}]\label{th_bicirculant_s2}
Every connected bicirculant graph whose vertices are incident to
at most two spokes is hamiltonian, except for the complete graph
$K_2$ and the Alspach generalized Petersen graphs.
\end{theorem}

\begin{theorem}[\protect{\cite[Theorem 1.3]{BoPiZi2025_2ndpaper}}]\label{th_m/2}
Let  $G=B(m; R,  S, T)$ be a connected bicirculant graph with $m$ even, $m\ge 4$, $|S|\ge 3$ and $m/2\in R, T$.
Then $G$ is hamiltonian.
\end{theorem}

\begin{theorem}[\protect{\cite[Theorem 1.4]{BoPiZi2025_2ndpaper}}]\label{th_main}
Let $G=B(m; R, S, T)$ be a connected bicirculant graph  with $m\ge 3$ and $|S|\ge 3$.
Assume that the connected components of $H(m; S)$ are hamiltonian and that $R\cup T$ contains at least one element that is not coprime to
$\gcd(m, S)$ in the case where $\gcd(m, S)>1$. Then $G$ is hamiltonian.
\end{theorem}

The following result will also be employed in our constructions.

\begin{lemma}[\protect{\cite[Lemma 3.1]{BoPiZi2025_2ndpaper}}]\label{lemma_bicirculant_m<5}
All connected bicirculant graphs of order $2m$ with $m\le 5$ are hamiltonian, except for the  complete graph $K_2$ and the 
Petersen graph $G(5,2)$. 
\end{lemma}
\smallskip

Now we review some properties of bicirculants, see \cite{GT} or \cite{BoPiZi2025_2ndpaper}. Recall that we always assume that $0 \in S$ in a bicirculant  $B(m;R,S,T)$.  
We will use the following notation.
Let $A=\{a_1,a_2,\dots,a_\ell\}$ be a set and let $i$ be an integer. We define $A-i=\{a-i \ | \ a \in A \}$ and $A/i= \{a/i \ | \ a \in A \}$. Moreover,  the notation $\gcd(A)$ should be understood as $\gcd(a_1,a_2,\dots,a_\ell)$ and $\gcd(i,A)$ as $\gcd(i,\gcd(A))$.

\begin{proposition}  \label{prop:connected}
A bicirculant graph $B(m;R,S,T)$  is connected if and only if $\gcd(m,R,S,T)=1$. 
\end{proposition}

In the case that a bicirculant graph is disconnected, it is composed of isomorphic connected components. 

\begin{proposition}\label{pro:d_connected_components} 
Let $G = B(m;R,S,T)$. Suppose $\delta = \gcd(m,R,S,T)>1$. 
Then $G$ is a disjoint union of $\delta$   isomorphic graphs $G_0, \dots G_{\delta-1}$ such that $u_i \in G_i$, $i = 0,\dots,\delta-1$. Moreover, each graph $G_i$  is connected and isomorphic to the graph $B(m/\delta;R/\delta,S/\delta,T/\delta)$. 
\end{proposition}

Different parameters may give the same bicirculants. In particular, it is easy to verify that the following result holds.

\begin{lemma}  \label{lemma:iso}
Let $G = B(m;R,S,T)$ be a bicirculant and let $c \in S$. Then the graph $B(m;R,S-c,T)$ is isomorphic to the graph $G$.
\end{lemma}

\section{Uniform and non-uniform representations} \label{sec_uniform_rep}

In this section we consider connected bicirculants of form  $B(m; \{a, -a\}, S, \{b, -b\})$  with $a, b \ne m/2$ and $|S| \ge 2$. For convenience we will denote such a 
bicirculant simply by $B(m; a, S, b)$. In \cite{BoPiZi2025_2ndpaper} we presented a construction of a hamilton cycle in a graph $B(m; a, S, b)$ using hamilton cycles in the 
connected components of its subgraph $H(m; S)$. That construction relies on a representation that can be formed for graphs $B(m; a, S, b)$ for which at least one of the integers $a$ 
or $b$ is not coprime to $\gcd(m, S)$.
This representation   will be called a \emph{uniform representation} of $B(m; a, S, b)$ and can be described as follows; see \cite{BoPiZi2025_2ndpaper}.

\smallskip

Let $G=B(m; a, S, b)$ be a connected bicirculant graph with $|S|\geq 2$, $\gcd(m, S)>1$,   $a,b\ne m/2$,  and at least one of the integers $a$, $b$, say $b$,  not coprime to $\gcd(m,S)$. From these assumptions it follows also that $m > 5$. We denote with $K$ the subgraph obtained from $G$ by removing the edges of type $a$ from $G$, that is,  $K$ is the bicirculant graph $B(m; \emptyset,S, \{b,-b\})$. We set $\lambda=\gcd(m, S, b)-1$, $\mu=\gcd(m, S)/\gcd(m, S, b)-1$ and denote with $K_0,\ldots, K_{\lambda}$ the connected components of $K$. By the assumptions, we have $\lambda >0$.

The connected components of $K$ form a partition of the components of the cyclic Haar graph $H(m; S)$ induced by the spokes of $G$.
More specifically, each component of $K$ consists of $(\mu+1)$ components of $H(m; S)$ that are connected by edges of type $b$.  
For $0\le j\le\lambda$, we denote the connected components of $H(m; S)$ that are contained in $K_j$ by $H_{i,j}$ ($0\le i\le\mu$). 
Moreover,  $H_{0,0}$ is the component of $H(m; S)$ containing the vertex $u_0$, and consequently $K_0$ is the component of $K$ containing $u_0$. 

The graph $G$ can be represented as follows: arrange the components $H_{i,j}$ with the same index $i$ in the same $i$-th row, and the components $H_{i,j}$ with the same index $j$ in the same $j$-th column. In this setting, the outer vertices in $H_{i, j}$ are adjacent to the outer vertices in $H_{i, j+1}$ through edges of type $a$ for $0\le j\le\lambda-1$, and the inner vertices in $H_{i, j}$ are adjacent to the inner vertices in $H_{i+1, j}$ through edges of type $b$ for $0\le i\le\mu$ (the subscript 
$i$ is considered modulo $\mu+1$). We note that the outer vertices in $K_0$ are also adjacent to the outer vertices of $K_{\lambda}$.

The above representation of $G$ led us in \cite{BoPiZi2025_2ndpaper} to call the component $K_j$  `the vertical $j$-th component' of $G$. In this paper, we call 
such a representation a \emph{uniform representation} or a $(\lambda,\mu)$-\emph{uniform representation} of $G$. Roughly speaking, a $(\lambda,\mu)$-uniform representation of $G$ consists of $(\lambda+1)$ vertical components $K_j$ of $K=B(m; \emptyset,S, \{b,-b\})$, each of which consists of $(\mu+1)$ components of $H(m; S)$. This property makes the representation `uniform'. Figure \ref{fig_ham_cycle_lambda_mu_even_H(m,S)_hamiltonian} shows a bicirculant graph that has a $(\lambda,\mu)$-uniform representation with  $\lambda=3$ and $\mu=2$. The components $H_{i,j}$ of $H(m; S)$ are represented by circles that are  arranged in a rectangular grid consisting of $\lambda+1$ vertical columns, each of which contains $\mu+1$ components of $H(m; S)$.

\smallskip

Uniform representations cannot be defined when both integers $a, b$ are coprime to $\gcd(m, S)$. In this case, we  define non-uniform representations, as we explain below.

\smallskip

Let $G=B(m; a, S, b)$ be a connected bicirculant graph with  $m>5$,
$|S|\geq 2$, $\gcd(m, S)>1$,  $a,b\ne m/2$,  and both integers $a$, $b$, coprime to $\gcd(m,S)$. 
Let $\lambda, \mu$ be positive integers, and let $\rho$ be a nonnegative integer such that $0\le\rho<\lambda$ 
and $(\rho+1)(\mu+1)+(\lambda-\rho)\mu = \gcd(m,S)$. 
We say that $G$ has a $(\lambda,\mu,\rho)$-\emph{non-uniform representation} if it can be represented as a vertex-disjoint union of $(\lambda+1)$ subgraphs $K_j$  with $0\le j\le\lambda$, which partition the components of $H(m; S)$ as follows.
There are $(\rho+1)$ subgraphs $K_j$ that contain $(\mu+1)$ components of $H(m; S)$; specifically, $K_j$ for $0\le j\le\rho$ contains the components $H_{i, j}$ with $0\le i\le\mu$. Each of the remaining subgraphs $K_j$, for $\rho+1\le j\le\lambda$,  contains $\mu$ components $H_{i, j}$ with $0\le i\le\mu-1$.

The subgraphs $K_j$ are called the \emph{vertical components} of $G$, in analogy to uniform representations.
Moreover, the adjacencies between outer and inner vertices of the components $H_{i,j}$ are the same as in the case of uniform representations, with the exception of 
the adjacencies between the inner vertices of $H_{0,j}$ and the inner vertices of $H_{\mu, j}$. 
In detail, the inner vertices in $H_{i,j}$ are adjacent to the inner vertices in $H_{i+1,j}$ through edges of type $b$, for $0\le i\le\mu-2$ and $0\le j\le\lambda$, and also for $i=\mu-1$ and $0\le j\le\rho$; the outer vertices in $H_{i,j}$ are adjacent to the outer vertices in $H_{i,j+1}$ through edges of type $a$, for $0\le i\le\mu-1$ and $0\le j\le\lambda-1$,  and also for $i=\mu$ and $0\le j\le\rho-1$.  

Figure \ref{fig_extension_step1} shows a bicirculant graph that has a $(\lambda, \mu, \rho)$-non-uniform representation with $\lambda=\mu=3$ and $\rho=2$.
The components $H_{i,j}$ of $H(m; S)$ are represented by circles that are  arranged in a non-rectangular grid consisting of $\rho+1$ vertical columns, each with $\mu+1$ components of $H(m; S)$, and of $\lambda-\rho$ vertical columns, each with $\mu$ components of $H(m; S)$. The vertical columns are the vertical components $K_j$, with $0\le j\le\lambda$, that define the non-uniform representation of $G$. Equivalently, the components $H_{i,j}$  of $H(m; S)$ are  arranged in a non-rectangular grid consisting of $\mu$ rows with $\lambda+1$ components and one row with $\rho + 1$ components. 

\smallskip

The next lemma shows that every graph  $B(m; a, S, b)$ satisfying the assumptions introduced above admits an $(\lambda,\mu,\rho)$-non-uniform representation with $\lambda\ge 1$,  $\mu> 1$ and $0\le\rho<\lambda$, provided that $b\not\equiv\pm a\pmod{\gcd(m,S)}$. 
Moreover, we determine the exact values of the parameters $\lambda$,  $\mu$ and $\rho$.

\begin{lemma}\label{lem_nonuniform_representation}
Let  $G=B(m; a, S, b)$ be a connected bicirculant graph with 
$m>5$, $|S|\ge 2$, $\gcd(m,S)>1$, $a,b \ne m/2 $ both coprime to $\gcd(m, S)$ and $b\not\equiv\pm a\pmod{\gcd(m,S)}$. 
Then $G$ admits a $(\lambda, \mu,\rho)$-non-uniform representation with $\lambda \ge 1$,  $1<\mu<\gcd(m,S)$   and $0\le\rho<\lambda$. 

The parameters $\lambda,\mu, \rho$ for such a representations can be obtained as follows. 
Let $h$ be the solution of the congruence $b\equiv ha\pmod{\gcd(m,S)}$ with the property $1< h<\gcd(m,S)-1$. Let $h^*=\min\{h, \gcd(m,S)-h\}$. Then $\lambda=h^*-1$ and $\mu,\rho+1$ are the quotient and the remainder, respectively,  in the division of $\gcd(m,S)$ by $h^*$, i.e.,  $\gcd(m,S)=\mu h^*+\rho+1$.
\end{lemma}

\begin{remark}
Notice that the congruence $b\equiv ha\pmod{\gcd(m,S)}$ has a unique solution $h$ modulo $\gcd(m,S)$ since $a$ is coprime to $\gcd(m,S)$.  Moreover, because $b$ is also coprime to $\gcd(m,S)$, it follows that $h$ is coprime to $\gcd(m,S)$ as well. The assumption $b\not\equiv\pm a\pmod{\gcd(m,S)}$ further implies that $1< h<\gcd(m,S)-1$. Hence, $\gcd(m,S)$ is not divisible by $h$. 
\end{remark}

\begin{proof}
We set $g=\gcd(m,S)-1$  and let $h$ be the solution of the congruence $b\equiv ha\pmod{g+1}$. Let $h^*=\min\{h, \gcd(m,S)-h\}$.
Without loss of generality we can assume that $h^*=h$ (since  we can interchange $b$  and $-b$). Set $\lambda = h-1$.
Since $h>1$, we can  set $\mu$ and $\rho+1$ as the the quotient and the remainder in the division of $g+1$  by $h$, respectively. Consequently, $g+1=\mu h^*+\rho+1$. 
Note that $\mu\ge 1$ and  $1\le\rho+1<h$,  since $h<g+1$ and 
$g+1$ is not divisible by $h$.

We denote with $H_0,\ldots, H_g$ the connected components of $H(m;S)$; $H_0$ is the component containing the vertex $u_0$.
In this notation, the outer vertices of $H_i$ are adjacent to the outer vertices of $H_{i+1}$ through the edges of type $a$,  and the inner vertices of $H_{i}$ are adjacent to the inner vertices of $H_{i+h}$ through the edges of type $b$,  since $b\equiv ha\pmod{g+1}$.  This congruence of $b$ modulo $(g+1)$ defines a partition of the components of $H(m; S)$.  More specifically,  we can partition the components of $H(m; S)$ into the sets $\mathcal H_0,\ldots,\mathcal H_{\mu}$ that are defined as follows:
$$
\begin{array}{lll}
\mathcal H_i=&\{H_{ih+j} : 0\le j\le h-1\}, & 0\le i\le\mu-1;\\
\mathcal H_{\mu}=&\{H_{\mu h+j} : 0\le j\le\rho\}.&\\
\end{array}
$$

\noindent Each $\mathcal H_i$, with $0\le i\le\mu-1$, consists of $h$ connected components of $H(m; S)$, while $\mathcal H_{\mu}$ consists of $(\rho+1)$ components. 
Notice that the element $H_{(i+1)h-1}$ 
in $\mathcal H_i$ is connected to $H_{(i+1)h}$ in $\mathcal H_{i+1}$ through edges of type $a$, since the components
$H_0,\ldots, H_g$ are cyclically connected through edges of type $a$.

In this setting we denote the components of $H(m; S)$ as follows: we set $H_{i,j}=H_{ih+j}\in\mathcal H_i$ for $0\le i\le\mu-1$ and $0\le j\le h-1$,
and $H_{\mu,j}=H_{\mu h+j}\in\mathcal H_{\mu}$ for $0\le j\le\rho$. 

In this notation one can see that $H_{i,j}$ and $H_{i,j+1}$ in $\mathcal H_i$ are connected through edges of type $a$, for every $0\le i\le\mu-1$ and $0\le j\le h-2$,
and also for $i=\mu$ and $0\le j\le\rho-1$; the components $H_{i,j}\in\mathcal H_i$ and  $H_{i+1,j}\in\mathcal H_{i+1}$ are connected through edges of type $b$,
for every $0\le i\le\mu-2$ and $0\le j\le h-1$, and also for $i=\mu-1$ and $0\le j\le\rho$. 
In other words, the partition $\mathcal H_0,\ldots,\mathcal H_{\mu}$ provides
a $(\lambda, \mu,\rho)$-non-uniform representation of $G$ with $\lambda=|\mathcal H_i|-1=h^*-1\ge 1$ for some $i<\mu$, $\mu\ge 1$ and $0\le\rho<\lambda$.  

It remains to show that $\mu>1$.  Suppose to the contrary that $\mu=1$. 
Then $g+1=h+\rho + 1$ and therefore $g+1-h = \rho +1$. Recall that $\rho+1<h$. Since $h=h^*=\min\{h,g+1-h\}$, we get a contradiction.
\end{proof}

Observe that a $(\lambda,\mu,\rho)$ representation of a bicirculant graph $B(m;a,S,b)$ is not unique; the proof of Lemma \ref{lem_nonuniform_representation}  provides  four distinct non-uniform representations for such a graph, at least two of them having $\mu >1$.
The notion of a $(\lambda,\mu,\rho)$-non-uniform representation generalizes the concept of a $(\lambda,\mu)$-uniform representation of a bicirculant graph $B(m;a,S,b)$, which is defined only when at least one of  $a$ or $b$  is not coprime to $\gcd(m,S)$,   if we allow  $\rho=\lambda$.

\section{Constructions for bicirculants with a non-uniform representation }\label{sec_extension_method1}

In \cite[Lemma 3.2 ]{BoPiZi2025_2ndpaper}, we presented a construction for hamilton cycles that can be applied to bicirculant graphs $B(m; a, S, b)$ for which at least one of $a$ or $b$ is not coprime to $\gcd(m, S)$, i.e., it can be applied to bicirculants with a uniform representation. 
In this section, we extend the construction presented in the proof of \cite[Lemma 3.2 ]{BoPiZi2025_2ndpaper} for bicirculant graphs $B(m; a, S, b)$ that have a $(\lambda,\mu)$-uniform representation with $\mu >0$ to bicirculant graphs $B(m; a, S, b)$ that have a $(\lambda,\mu,\rho)$-non-uniform representation with $\mu>1$. 
We will report the statement of \cite[Lemma 3.2 ]{BoPiZi2025_2ndpaper} later in this section together with some useful remarks.  Before that, we introduce the notation that will be used.

\smallskip

Let $G=B(m; a, S, b)$ be a bicirculant graph that has a $(\lambda, \mu, \rho)$-non-uniform representation with $\mu>1$. Let $K_j$, with $0\le j\le\lambda$, be  the vertical components that partition the connected components of the subgraph $H(m; S)$ induced by the spokes of $G$. Following the results in Section \ref{sec_uniform_rep}, each vertical component $K_j$, with $0\le j\le\rho$, contains the components $H_{i, j}$ of $H(m; S)$ with $0\le i\le\mu$, and each of the remaining vertical components $K_j$, with $\rho+1\le j\le\lambda$,  contains $\mu$ components $H_{i, j}$ with $0\le i\le\mu-1$. Moreover, $H_{0,0}$ is the component of $H(m; S)$ that contains the vertex $u_0$. If necessary, we replace $b$ with $-b$, so that the vertex $v_b$ is in the component $H_{1,0}$. We need to do this when $b\equiv ha\pmod{\gcd(m,S)}$ with $h > \gcd(m,S)/2$; see the proof of Lemma \ref{lem_nonuniform_representation}.

By Lemma \ref{lem_ham_path_cyclic_Haar}, there exists a hamilton path $P$ in $H_{0,0}$.
As in \cite{BoPiZi2025_2ndpaper}, we can assume that $P$ is a hamilton path with origin in $v_s$, for some $s \in \ZZ_n$, and terminus in $u_0$; we denote with $u_z$ an arbitrary vertex of $P$ different from $u_0$ and with $v_h$, $v_k$ the vertices adjacent to $u_z$ in $P$. Moreover, we denote by $u_t$ the vertex adjacent to $v_s$ in $P$.
For convenience of notation, the remaining vertices of $P$ will be simply denoted with $u_i$, $v_i$ instead of $u_{x_i}$, $v_{y_i}$.
More specifically, we set 

$$P=(v_s, u_t,   \ldots,   v_h,  u_z,  v_k,  \ldots,   u_1,  v_0 , u_0).$$

We will use the notation $x\,P\,y$ to denote a path from the vertex $x$ to the vertex $y$.
It may happen that $v_0=v_k$ and $u_1=u_z$ or $v_s=v_h$ and $u_t=u_z$, but for $|S|\ge 4$ 
we can assume that the vertices $u_0, u_1, u_z, u_t$ and $v_0, v_s, v_h, v_k$ are pairwise distinct. Moreover, 
it is understood that $u_0$ and $v_s$ are adjacent if the connected components of $H(m; S) $ are hamiltonian.
In this case, we denote with $C$ the hamilton cycle in $H_{0,0}$ that we obtain by adding the edge $u_0v_s$ 
to $P$, that is, $C=P\cup u_0v_s$. 

The connected components $H_{i, j}$ of $H(m; S)$ contain copies of $P$. We denote with $P_{i,j}$  the copy of the path $P$ in the component $H_{i,j}$; $P_{0,0}$ corresponds to the path $P$. The path $P_{i,j}$ can be obtained from $P$ by adding $ib+ja\pmod m$ to the subscripts of the vertices in $P$ (see the proof of Lemma \ref{lem_nonuniform_representation}).

Accordingly, the vertices in $P_{i,j}$ will be denoted with $u^{i,j}_x$, $v^{i,j}_x$, where $u^{i,j}_x=u_{x+ib+ja}$, $v^{i,j}_x=v_{x+ib+ja}$.  The outer vertices $u^{i,j}_x$ in $P_{i,j}$ are adjacent to the outer vertices $u^{i,j+1}_x$ in $P_{i, j+1}$ through 
the edges of type $a$; the inner vertices $v^{i,j}_x$ in $P_{i,j}$ are adjacent to the inner vertices $v^{i+1,j}_x$ in $P_{i+1, j}$ 
through the edges of type $b$. We stress that a non-uniform representation does not include the edges from the inner vertices in $H_{0,j}$ to the inner vertices in $H_{\mu',j}$, where $\mu'=\mu$ if $0\le j\le\rho$, while $\mu'=\mu-1$ if $\rho+1\le j\le\lambda$. 

For simplicity of notation, we will denote the copy of the subpath $u_x\,P\,v_y$ from the vertex $u_x$ to the vertex $v_y$ of $P$  that is contained in $H_{i, j}$ with $u_x\,P_{i,j}\,v_y$ or with  $u_x\,P\,v_y$  in $H_{i, j}$;  
it is understood that the vertices $u_x$, $v_y$ in the notation $u_x\,P_{i,j}\,v_y$, or in the notation $u_x\,P\,v_y$ in $H_{i,j}$,  correspond to the vertices $u^{i,j}_x$, $v^{i,j}_y$.  

With this notation, the reader can refer to Figure \ref{fig_ham_cycle_lambda_mu_even_H(m,S)_hamiltonian} showing the construction we give in the proof of  \cite[Lemma 3.2]{BoPiZi2025_2ndpaper} for a connected bicirculant graph having a
uniform representation; in particular a $(\lambda, \mu)$-uniform representation with $\lambda=3$ and $\mu=2$. Below, we also report the statement of \cite[Lemma 3.2]{BoPiZi2025_2ndpaper}.

\begin{lemma}[\protect{\cite[Lemma 3.2]{BoPiZi2025_2ndpaper}}]\label{lemma_construction_2nd_paper}
Let $G=B(m; a, S, b)$ be a connected bicirculant graph with $m>5$, $|S|\geq 2$, $\gcd(m, S)>1$, $a,b\ne m/2$, and 
$\gcd(m, S, b) >1$.  
Assume that the connected components of $H(m; S)$ are hamiltonian.  Then $G$ is hamiltonian.
\end{lemma}

\begin{figure}[h]
\begin{center}
\includegraphics[width=12cm]{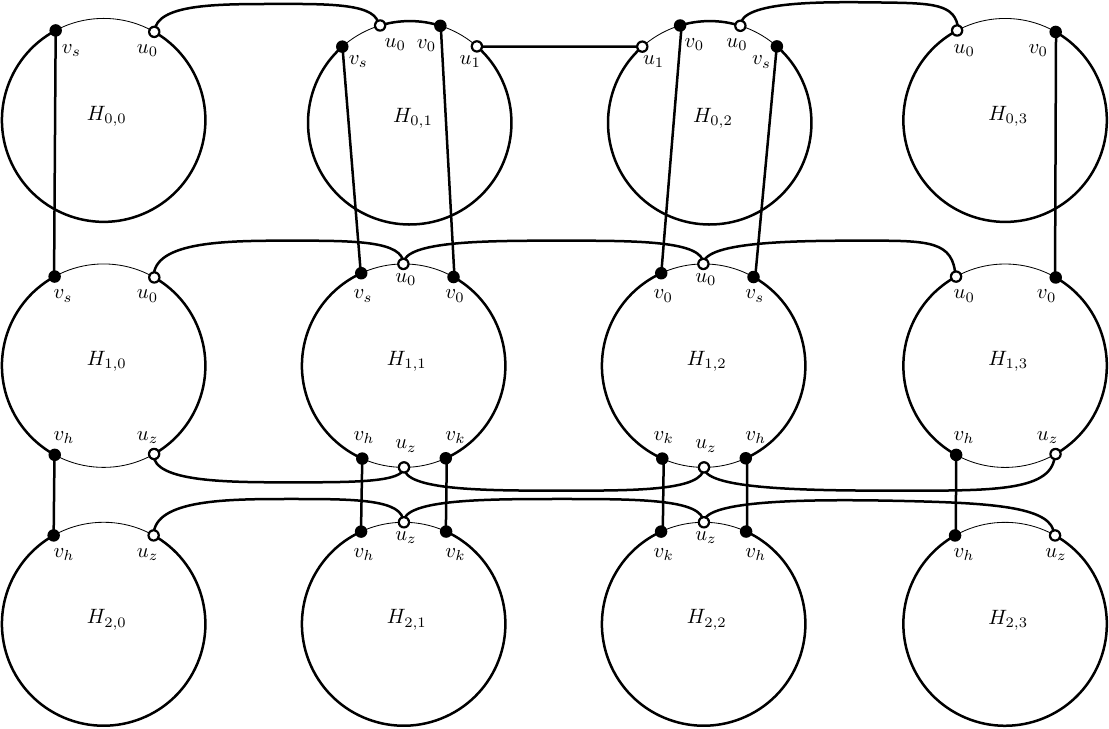}
\end{center}
\caption{The construction of a hamilton cycle we give in the proof of Lemma  \cite[Lemma 3.2 ]{BoPiZi2025_2ndpaper} for bicirculant graphs having a uniform representation. The graph in the figure has a $(\lambda, \mu)$-uniform representation with
$\lambda=3$ and $\mu=2$.}\label{fig_ham_cycle_lambda_mu_even_H(m,S)_hamiltonian}
\end{figure}

 The following remark on the construction described in the proof of \cite[Lemma 3.2]{BoPiZi2025_2ndpaper} is in order. It will help in the proof of Lemma \ref{lem_extension_method1}.

\begin{remark}\label{rem_construction}
Figure \ref{fig_ham_cycle_lambda_mu_even_H(m,S)_hamiltonian} shows that the construction described 
in the proof of \cite[Lemma 3.2]{BoPiZi2025_2ndpaper} for the case $\mu >0$ does not use the edges connecting the outer vertices in $K_0$ to the outer vertices in $K_{\lambda}$, nor the edges from the inner vertices in $H_{\mu, j}$ to the inner vertices in $H_{0, j}$ for every $0\le j\le\lambda$. 
This fact allows us to apply the construction in the proof of  \cite[Lemma 3.2]{BoPiZi2025_2ndpaper} to any
subgraph of $B(m; a, S, b)$ that is composed of a certain number of connected components of $H(m; S)$, say $(\lambda'+1)(\mu'+1)$ with $0<\lambda'\le\lambda$ and $0<\mu'\le\mu$, 
that can be arranged in a rectangular grid with $(\mu'+1)$ rows consisting of $(\lambda'+1)$ components of  $H(m; S)$ or, equivalently, in a rectangular grid with $(\lambda'+1)$ columns consisting of $(\mu'+1)$ components of  $H(m; S)$. More specifically, 
in a uniform representation of $B(m; a, S, b)$, we can select two positive integers, say $\lambda'$ and $\mu'$, and two non-negative indices, say $i'$ and $j'$,  such that $i'+\mu'\le\mu$, $j'+\lambda'\le\lambda$, and consider the components $H_{i, j}$ with $i'\le i\le i'+\mu'$ and $j'\le j\le j'+\lambda'$. 

The adjacencies among the selected components are inherited from the uniform representation of $G$ as follows: for $i'\le i\le i'+\mu'$, the outer vertices in $H_{i, j}$ are adjacent to  the outer vertices in $H_{i, j+1}$ for every $j'\le j\le j'+\lambda'-1$; for $j'\le j\le j'+\lambda'$, the inner vertices in $H_{i, j}$ are adjacent to  the inner vertices in $H_{i+1, j}$ for every $i'\le i\le i'+\mu'-1$. 

It does not matter whether the the inner vertices in $H_{i'+\mu', j}$ are adjacent to the inner vertices in $H_{i', j}$, or the
outer vertices in $K_{j'+\lambda'}$ are adjacent to the outer vertices in $K_{j'}$: 
the construction in the proof of \cite[Lemma 3.2]{BoPiZi2025_2ndpaper}  works the same.
For this reason, the construction in the proof of  \cite[Lemma 3.2]{BoPiZi2025_2ndpaper} also works in the case where the graph $B(m; a, S, b)$ has a $(\lambda,\mu,\rho)$-non-uniform representation and we consider $i'+\mu'<\mu$, $j'+\lambda'\le\lambda$ and $\mu'>0$. It provides a hamilton cycle in the selected subgraph of the graph $B(m; a, S, b)$.
\end{remark}


We are now ready to present Lemma \ref{lem_extension_method1}, which extends the construction we give in the proof of \cite[Lemma 3.2]{BoPiZi2025_2ndpaper}.
In the construction described in Lemma \ref{lem_extension_method1}, we will also consider the hamilton cycle $C=P\cup u_0v_s$ and the following paths: 
the hamilton path $u_0\,C\,v_0=C-u_0v_0$ 
that is obtained by removing the edge $u_0v_0$ from the cycle $C$; 
the hamilton path  $u_z\,C\,v_x=C-u_zv_x$, with $x \in \{h,k\}$, that is obtained by removing the edge $u_zv_x$ from the cycle $C$;
the path $v_s\,C\,v_0=C-u_0$ that is obtained by deleting the vertex $u_0$ from $C$;
the path $v_h\,C\,v_k=C-u_z$ that is obtained by deleting the vertex $u_z$ from $C$. 
We will denote with  $u_0\,C_{i,j}v_0$, $u_z\,C_{i,j}v_x$, $v_s\,C_{i,j}\,v_0$ and $v_h\,C_{i,j}\,v_k$  the 
copies of $u_0\,C\,v_0$, $u_z\,C\,v_x$, $v_s\,C\,v_0$ and $v_h\,C\,v_k$  that are contained in the component $H_{i, j}$ of $H(m; S)$.

Figures \ref{fig_extension_step1} and \ref{fig_extension_step2} show the construction described in the proof of Lemma \ref{lem_extension_method1} in the case where the bicirculant graph has a $(\lambda, \mu, \rho)$-non-uniform representation with 
$\lambda=3$, $\mu=3$ and $\rho=2$.

\begin{lemma}\label{lem_extension_method1}
Let $G=B(m; a, S, b)$ be a connected bicirculant graph with $|S|\ge 2$ that admits a $(\lambda,\mu,\rho)$-non-uniform representation  with $\lambda\ge 1$, $\mu> 1$ and $0\le\rho<\lambda$. 
Assume the connected components of $H(m; S)$ are hamiltonian. 
Then the following holds.
\begin{itemize}
\item[(i)] If $\rho=0$, then $G$ has a hamilton path from the vertex $u_0$ to the vertex $u_b$.
\item[(ii)]
If $\rho>0$, then $G$ is hamiltonian.  
\end{itemize}
\end{lemma}

\begin{proof}
Given a $(\lambda,\mu,\rho)$-non-uniform representation of $G$ with $\mu> 1$, let $\hat G$ be the subgraph obtained from $G$ by removing the components $H_{\mu, j}$ for $0\le j\le\rho$.  We can say that $\hat G$ is the vertex-disjoint union of the vertical components $\hat K_j$, with $0\le j\le\lambda$, where $\hat K_j=K_j$ for $\rho+1\le j\le\lambda$,  while $\hat K_j=K_j-H_{\mu,j}$ for $0\le j\le\rho$.  Equivalently, according to Remark \ref{rem_construction}, $\hat G$ is obtained from 
$G$ by setting $\lambda'=\lambda$, $\mu'=\mu-1$ and $i'=j'=0$. By the same remark, we can apply the construction described in the proof of  \cite[Lemma 3.2]{BoPiZi2025_2ndpaper} and find a hamilton cycle, say $\hat C$, in $\hat G$. 
The cycle $\hat C$ covers all the vertices of $G$ with the exception of the vertices in $\cup^{\rho}_{j=0} H_{\mu, j}$.
In the following, we describe how to expand the cycle $\hat C$ into a hamilton cycle in $G$.

\begin{figure}
    \centering
    \includegraphics[width=0.9\linewidth]{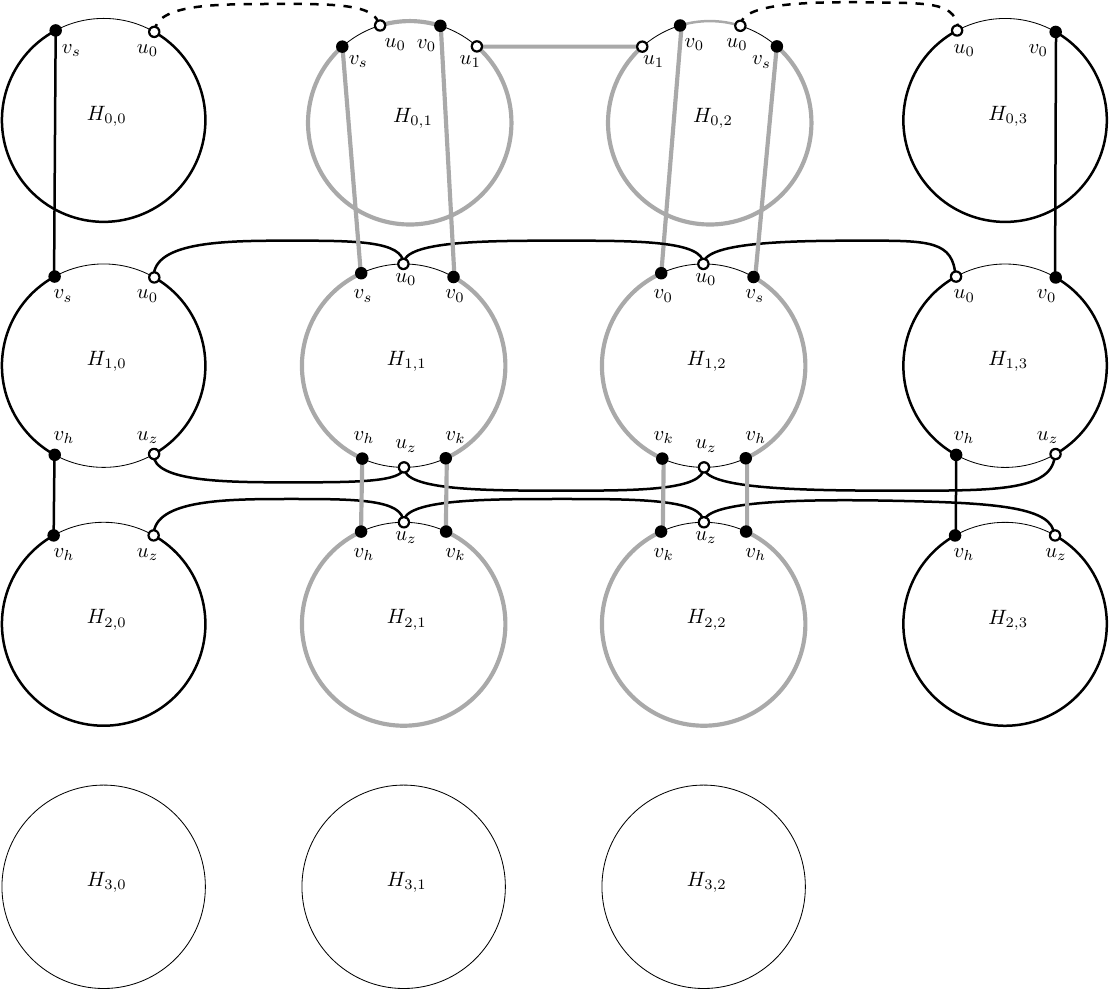}
    \caption{Step 1 in the proof of Lemma \ref{lem_extension_method1}. The figure shows a bicirculant graph that has a $(\lambda, \mu, \rho)$-non-uniform representation with $\lambda=\mu=3$ and $\rho=2$. The thick curves that are highlighted on the $H_{i,j}$ represent the subpaths of $C_{i,j}$ that are used in Lemma \ref{lem_extension_method1} to construct a hamilton cycle $\hat C$ in the subgraph $\hat G$ that is obtained by deleting the components $H_{\mu,j}$ with $0\le j\le\rho$ from $G$. The cycle $\hat C$ is the join of the paths $u^{0,0}_0\,P'\,u^{0,\lambda}_{0}$ and $u^{0,1}_{0}\,P^*\,u^{0,\lambda-1}_{0}$ through the edges $u^{0,0}_0u^{0,1}_0$ and $u^{0,\lambda-1}_{0}u^{0,\lambda}_{0}$. The path $u^{0,0}_0\,P'\,u^{0,\lambda}_{0}$ is represented by the thick edges in black, the path $u^{0,1}_{0}\,P^*\,u^{0,\lambda-1}_{0}$ is defined by the thick edges in gray, and the dashed black edges are  the edges $u^{0,0}_0u^{0,1}_0$ and $u^{0,\lambda-1}_{0}u^{0,\lambda}_{0}$.}
 \label{fig_extension_step1}
\end{figure}

As defined in the proof of \cite[Lemma 3.2] {BoPiZi2025_2ndpaper}, the cycle $\hat C$ is the join of the paths $u^{0,0}_0\,P'\,u^{0,\lambda}_{\alpha}$ and $u^{0,1}_{0}\,P^*\,u^{0,\lambda-1}_{\alpha}$ through the edges $u^{0,0}_0u^{0,1}_0$ and $u^{0,\lambda-1}_{\alpha}u^{0,\lambda}_{\alpha}$ (we recall that $u_{\alpha}=u_0$ if $\lambda$ is odd,  $u_{\alpha}=u_1$ if $\lambda$ is even, and for $\lambda=1$ we have that $\hat C$ consists of the the path $u^{0,0}_0\,P'\,u^{0,\lambda}_{\alpha}=u^{0,0}_0\,P'\,u^{0,1}_{0}$ together with the edge $u^{0,0}_0u^{0,1}_{0}$). 
In Figure \ref{fig_extension_step1}, the path $u^{0,0}_0\,P'\,u^{0,\lambda}_{\alpha}$, with $u_{\alpha}=u_0$, is given by the thick  edges in black, while the thick edges in gray give the path $u^{0,1}_{0}\,P^*\,u^{0,\lambda-1}_{\alpha}$, and the dashed black edges are  
the edges $u^{0,0}_0u^{0,1}_0$ and $u^{0,\lambda-1}_{\alpha}u^{0,\lambda}_{\alpha}$.

If $(\mu-1)$ is even, then the path $u^{0,0}_0\,P'\,u^{0,\lambda}_{\alpha}$ contains the path $u_z\,C_{\mu-1,0}\,v_h$.
For $\lambda>1$,  the path $u^{0,1}_{0}\,P^*\,u^{0,\lambda-1}_{\alpha}$ contains the path 
$v_h\,C_{\mu-1,j}\,v_k$ for every $1\le j\le\rho$.

If $(\mu-1)$ is odd, then the path $u^{0,0}_0\,P'\,u^{0,\lambda}_{\alpha}$ contains the path $v_s\,P_{\mu-1,0}\,u_0$.
For $\lambda>1$, the path $u^{0,1}_{0}\,P^*\,u^{0,\lambda-1}_{\alpha}$ contains the path $v_s\,P_{\mu-1,j}\,v_0$ for every $1\le j\le\rho$.

We will show the construction in the case where $\mu$ is odd and, consequently, $\mu-1$ is even. The construction in the case where $\mu-1$ is odd is analogous; we just need to replace the vertices $v_h$,  $u_z$,  $v_k$,  $v_s$,  $u_0$, $v_0$ with the vertices $v_s$,  $u_0$,  $v_0$,  $v_h$,  $u_z$, $v_k$, respectively.

In the following, we assume that $\mu-1$ is even and also consider $\rho>0$ (the case $\rho=0$ will be treated at the end of the proof). 
The reader can refer to Figure \ref{fig_extension_step2} for the construction that we are going to describe.
We consider the following paths spanning the vertices of $H_{\mu,j}$ with $0\le j\le\rho$: the copy of the path $P$ from $v_s$ to $u_0$ in $H_{\mu,0}$, the copy of $v_s\,C\,v_0$ together with the copy of the vertex $u_0$ in $H_{\mu, j}$ with $1\le j\le\rho-1$, and  in $H_{\mu,\rho}$ we take the copy of $u_0\,C\,v_0$ or the copy of $P$ according to whether $\rho$ is even or odd, respectively.
We set $v_{\beta}=v_0$ if $\rho$ is even, $v_{\beta}=v_s$ is $\rho$ is odd. We denote with $K'_{\rho}$ the subgraph of the vertical component $K_{\rho}$ consisting of the copies of the following paths: if $\rho$ is even, it consists of the copy of the path $v_s\,P\,u_1$ in $H_{0,\rho}$, the copy of the path $v_s\,P\,v_h$ in $H_{i,\rho}$ for $1\leq i\le\mu-2$, and the copy of the path $v_h\,P\,u_0$ in $H_{\mu-1,\rho}$; if $\rho$ is odd, it consists of the copy of the edge $u_0v_0$ in $H_{0,\rho}$, the copy of the path $v_k\,P\,v_0$ in $H_{i,\rho}$ for $1\leq i\le\mu-2$, and the copy of the path $v_k\,P\,u_0$ in $H_{\mu-1,\rho}$.

We modify the path $u^{0,1}_{0}\,P^*\,u^{0,\lambda-1}_{\alpha}$ to obtain a path $u^{0,1}_0\,P^{**}\,u^{\mu-1, \rho}_0$ connecting the vertices $u^{0,1}_0$ 
and $u^{\mu-1, \rho}_0$ as follows: in $\hat C$  we remove the edge 
$u^{\mu-1,\rho}_0v^{\mu-1,\rho}_{\beta}$,  and the edges $u^{\mu-1,j}_0v^{\mu-1,j}_s$,   $u^{\mu-1,j}_0v^{\mu-1,j}_0$ for every $1\le j\le\rho-1$, 
then we add the edges from the copies of $v_0$, $v_s$ in $H_{\mu-1, j}$ to the copies of $v_0$, $v_s$ in  
$v_s\,C_{\mu, j}\,v_0$ for $1\le j\le\rho-1$ (see the gray path in Figure \ref{fig_extension_step2}).

The path $u^{0,1}_0\,P^{**}\,u^{\mu-1, \rho}_0$ covers the vertices in the vertical component $K_j$, for $1\le j\le\rho-1$,
with the exception of the copies of the vertices $u_0$, $u_z$ in each $H_{i,j}$, with $1\le i\le\mu-1$ and $1\le j\le\rho-1$,
and  of the copy of $u_0$ in  $H_{\mu,j}$ with $1\le j\le\rho-1$. It also covers the vertices in  $K'_{\rho}$.
We denote with $U$ the set of vertices in $\cup^{\rho}_{j=1}K_j$ that are not covered by $u^{0,1}_0\,P^{**}\,u^{\mu-1, \rho}_0$.

Now, we replace the path $u^{0,0}_0\,P'\,u^{0,\lambda}_{\alpha}$ with three paths: the path 
$u^{0,0}_0\,P^{''}\,v^{\mu,\rho}_{\beta}$ from the vertex $u^{0,0}_0$ to the vertex $v^{\mu,\rho}_{\beta}$,
the path $u^{0,\lambda}_{\alpha}\,P^{'''}\,u^{\mu-1,\rho-1}_{0}$ from the vertex $u^{0,\lambda}_{\alpha}$
to the vertex $u^{\mu-1,\rho-1}_{0}$, and the path 
$u^{0,\lambda-1}_{\alpha}\,P^*\,v^{\mu-1,\rho}_{\beta}$
from the vertex $u^{0,\lambda-1}_{\alpha}$ to the vertex $v^{\mu-1,\rho}_{\beta}$. 
The first two paths are obtained from the path $u^{0,0}_0\,P'\,u^{0,\lambda}_{\alpha}$ as follows:
we remove the edge $u^{\mu-1,0}_0v^{\mu-1,0}_s$, then we add the edge from the copy of $v_s$ in $H_{\mu-1,0}$ to the copy of 
$v_s$ in $v_s\,P_{\mu, 0}\,u_0$, and also add the edges $u^{\mu-1,j}_0u^{\mu-1,j+1}_0$ with $0\le j\le\rho-2$ and 
$u^{\mu,j}_0u^{\mu,j+1}_0$ with $0\le j\le\rho-1$ (see the black paths in Figure \ref{fig_extension_step2}).
The path $u^{0,\lambda-1}_{\alpha}\,P^*\,v^{\mu-1,\rho}_{\beta}$ arises from $u^{0,1}_{0}\,P^*\,u^{0,\lambda-1}_{\alpha}$ after the removal of the edge $u^{\mu-1,\rho}_0v^{\mu-1,\rho}_{\beta}$ while constructing the path path $u^{0,1}_0\,P^{**}\,u^{\mu-1, \rho}_0$.
The union of the three paths covers the vertices in $U$ together with the vertices in the vertical components $K_0$ and $K_{\lambda}$.

We join the three above paths to the path $u^{0,1}_0\,P^{**}\,u^{\mu-1, \rho}_0$ through the following edges 
(see the dashed edges in Figure \ref{fig_extension_step2}): the edge
$u^{\mu-1,\rho-1}_0u^{\mu-1,\rho}_0$, the edge $v^{\mu-1,\rho}_{\beta}v^{\mu,\rho}_{\beta}$, and
the edges $u^{0,0}_0u^{0,1}_0$ and  $u^{0,\lambda-1}_{\alpha}u^{0,\lambda}_{\alpha}$,  which joined the paths 
$u^{0,0}_0\,P'\,u^{0,\lambda}_{\alpha}$ and $u^{0,1}_{0}\,P^*\,u^{0,\lambda-1}_{\alpha}$ in $\hat C$. 
We obtain a hamilton cycle in $G$, and the assertion follows in the case where $\rho>0$.

Finally, we consider the case $\rho=0$. In this case, the removal of the edge $u^{\mu-1,0}_0v^{\mu-1,0}_s$ makes $\hat C$ a path from $u^{\mu-1,0}_0$ to $v^{\mu-1,0}_s$, say $u^{\mu-1,0}_0\,\hat C\,v^{\mu-1,0}_s$, that covers all vertices of $G$ with the exception of the vertices in $H_{\mu, 0}$.  By joining 
the paths $u^{\mu-1,0}_0\,\hat C\,v^{\mu-1,0}_s$ and   $v_s\, P_{\mu, 0}\, u_0$ through the edge $v^{\mu-1,0}_sv^{\mu,0}_s$, we find a hamilton path
in $G$ from the vertex $u^{\mu-1,0}_0$ to the vertex $u^{\mu, 0}_0$, which yields a hamilton path from $u_0$ to $u_b$ since $u^{\mu-1,0}_0=u_{(\mu-1)b}$ and
$u^{\mu, 0}_0=u_{\mu b}$. Then the assertion follows for $\rho=0$ and, consequently, the assertion is true for $\lambda=1$ (for $\lambda=1$ we necessarily have $\rho=0$, since $0\le\rho<\lambda$).  

The above constructions are valid even in the case where some of the vertices that are used in the description of the paths are the same, i.e., 
if $v_0=v_k$ and $u_1=u_z$ and/or $v_s=v_h$ and $u_t=u_z$; this may happen when $|S|=2, 3$, while for $|S|\ge 4$ the vertices that are used in the description of the paths can be considered distinct.

\begin{figure}
    \centering
    \includegraphics[width=0.9\linewidth]{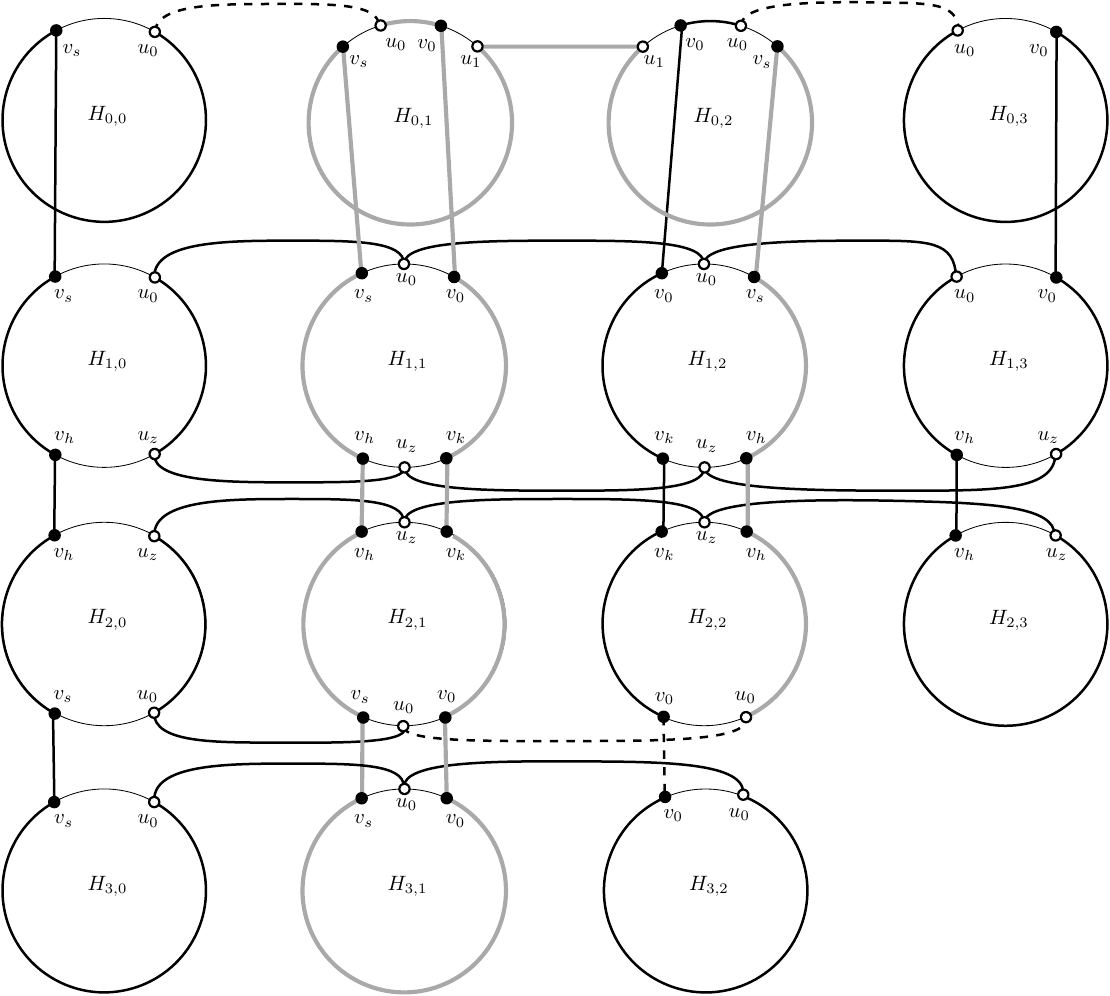}
    \caption{Step 2 in the proof of Lemma \ref{lem_extension_method1}. We extend the hamilton cycle $\hat C$ in $\hat G$ to hamilton cycle in $G$. First, we extend the path  $u^{0,1}_{0}\,P^*\,u^{0,\lambda-1}_{0}$ to a path $u^{0,1}_0\,P^{**}\,u^{\mu-1, \rho}_0$ connecting the vertices 
$u^{\mu-1,\rho}_0$ and $u^{\mu-1, \rho}_0$ (see the gray edges). Then, we replace the path $u^{0,0}_0\,P'\,u^{0,\lambda}_{0}$ with the paths
$u^{0,0}_0\,P^{''}\,v^{\mu,\rho}_{0}$,  $u^{0,\lambda}_{0}\,P^{'''}\,u^{\mu-1,\rho-1}_{0}$ and  $u^{0,\lambda-1}_{0}\,P^*
\,v^{\mu-1,\rho}_{0}$ (see the thick edges in black). Finally, we join the above paths through the following edges (see the dashed edges): 
$u^{\mu-1,\rho-1}_0u^{\mu-1,\rho}_0$,  $v^{\mu-1,\rho}_{0}v^{\mu,\rho}_{0}$, 
$u^{0,0}_0u^{0,1}_0$ and  $u^{0,\lambda-1}_{0}u^{0,\lambda}_{0}$. We obtain a hamilton cycle in $G$.}
    \label{fig_extension_step2}
\end{figure}

\end{proof}

\section{Further constructions}

In this section we study bicirculant graphs  of form $B(m; a, S, b)$ that have a $(\lambda, \mu, \rho)$-non-uniform representation with $\rho=0$  and $\mu >1$, or  their parameters satisfy the relation $b\equiv\pm a\pmod{\gcd(m, S)}$, that is, bicirculant graphs that are not covered by Lemma \ref{lem_extension_method1} (ii). 
The first main result of this section -- Lemma  \ref{lem_ab_coprime2}   --  deals with the bicirculants that  have a $(\lambda, \mu, \rho)$-non-uniform representation with $\rho=0$  and $\mu >1$. In order to prove Lemma  \ref{lem_ab_coprime2}, we need the preliminary Lemma  \ref{lem_rho=02}. We will also employ the $2$-hooked construction from \cite{BoPiZi2025} that, for the sake of completeness, we report in Proposition \ref{rem_2hooked2}.

\begin{lemma}\label{lem_rho=02}
Let  $G=B(m; a, S, b)$ be a connected bicirculant graph with $m>5$, $|S|\ge  3$, $\gcd(m,S)>1$, $a,b \ne m/2 $ both coprime to $\gcd(m, S)$ and  $b\not\equiv\pm a\pmod{\gcd(m,S)}$. 
Assume that $G$ has a $(\lambda, \mu,\rho)$-non-uniform representation  with $\rho< \lambda$.
Further to the above, assume that there exists a subset $S'$ of $S$  containing element $0$  such that $\delta=\gcd(m, a, S',b)>1$, and $\gcd(m,a, S'\cup\{c\},b)=1$ for some
$c\in S\smallsetminus S'$. 

Then every connected component of the bicirculant  $B(m; a, S', b)$ is  isomorphic to the  bicirculant $B(m; a/\delta, S'/\delta, b/\delta)$ where $a/\delta$, $b/\delta$ are coprime to $\gcd(m,S')/\delta$.
Moreover,  $\gcd(m,S')=\gcd(m,S)\delta$ and each connected component of   $B(m; a, S', b)$ has a $(\lambda', \mu',\rho')$-non-uniform representation with
$\lambda'=\lambda$, $\mu'=\mu$ and $\rho'=\rho$. %
\end{lemma}

\begin{proof}
Let us consider the partition $\mathcal H_0,\ldots,\mathcal H_{\mu}$  of the connected components of the graph $H(m;S)$ defined in the proof of Lemma \ref{lem_nonuniform_representation}, and the components $H_{i,j}$ of $H(m; S)$
belonging to $\mathcal H_i$, with $0\le i\le\mu$.

Each component $H_{i,j}$ splits into  $\gcd(m,S')/\gcd(m; S)$ components of $H(m; S')$, say $H_{i,j,k}$ with $0\le k\le \gcd(m,S')/\gcd(m; S)-1$.
By the arguments used in Lemma \ref{lem_nonuniform_representation},  the component $H_{i,j, k}$ is connected to the components $H_{i,j+1, k}$ and $H_{i+1,j, k}$ trough edges
of type $a$ and $b$,  respectively, and is also connected to the component $H_{i,j, k+1}$ trough edges of type $c$, since $\gcd(m,a, S'\cup\{c\},b)=1$.  
Notice that $H_{i,h-1, k}$ is connected to $H_{i+1,0, k}$ through edges of type $a$ -- see the remark in Lemma \ref{lem_nonuniform_representation} about $H_{i, h-1}$ and $H_{i+1, 0}$.
Consequently,  each $\mathcal H_i$ splits into  $\gcd(m,S')/\gcd(m; S)$ sets, say $\mathcal H_{i,k}=\{H_{i,j, k}: H_{i,j}\in\mathcal H_i\}$ with $0\le k\le (\gcd(m,S')/\gcd(m; S))-1$.
Because of the adjacencies between the vertices of the component $H_{i,j, k}$ with the vertices of the components $H_{i,j+1, k}$,  $H_{i+1,j, k}$ and  $H_{i,j, k+1}$, 
the sets $\mathcal H_{0,k}, \mathcal H_{1,k},\ldots, \mathcal H_{\mu,k}$ form a partition of the vertex set of the $k$-th component of $B(m; a, S', b)$, for every $0\le k\le (\gcd(m,S')/\gcd(m; S))-1$.
Then $\delta=\gcd(m,S')/\gcd(m; S)$, i.e.,   $\gcd(m,S')=\gcd(m,S)\delta$ since $B(m; a, S', b)$ has $\delta$ connected components. 
Moreover,  the sets $\mathcal H_{0,k}, \mathcal H_{1,k},\ldots, \mathcal H_{\mu,k}$ provide a $(\lambda', \mu',\rho')$-non-uniform representation of the 
$k$-th component of $B(m; a, S', b)$. 
In detail, $\lambda'=|\mathcal H_{i,k}|-1=|\mathcal H_i|-1=\lambda$ for some $i<\mu$, $\mu'=\mu$,  $\rho'=|\mathcal H_{\mu,k}|-1=|\mathcal H_{\mu}|-1=\rho$. 
All this also means that $a/\delta$, $b/\delta$ are coprime to $\gcd(m,S')/\delta$. The assertion follows.
\end{proof}

We describe in detail another construction that expands hamilton cycles or paths in the connected components of a subgraph to a hamilton cycle of the whole graph. The so-called $2$-hooked construction was defined in  \cite{BoPiZi2025} and it was used to construct hamilton cycles in the generalized rose window graphs.

\begin{proposition}[\protect{\cite[The $2$-hooked construction]{BoPiZi2025}}]\label{rem_2hooked2}

Let $B(m,  a, S'\cup\{c\}, b)$ be a connected bicirculant graph such that $\gcd(m, S'\cup\{c\} )>1$ and $\gcd(m, a, S', b)>1$.

Assume that each connected component of $B(m; a, S', b)$ has a hamilton cycle containing outer and inner edges,
and that it also has a hamilton path from the copy of $u_0$ to the copy of $u_b$ 
or a hamilton path from the copy of $v_0$ to the copy of $v_a$. 
Then the graph $B(m,  a, S'\cup\{c\}, b)$ is hamiltonian.

\end{proposition}

\begin{proof}
We set $G=B(m,  a, S'\cup\{c\}, b)$ and $G'=B(m; a, S', b)$. The following notation will be used to describe the $2$-hooked construction. The connected components of $G'$ will be denoted with $H'_0,\ldots, H'_{\gamma}$, where $\gamma=\gcd(m, a, S',b)-1$; 
$H'_0$ is the component containing $u_0$. For $1\leq i\leq \gamma$, the copies of the vertices $u_x$, $v_x$ of $H'_0$ that are contained in $H'_i$ are $u_{x+ci}$, $v_{x+ci}$ and they will be  be denoted with $u^i_x$, $v^i_x$; the vertices of $H'_0$ will be also denoted with $u^0_x$, $v^0_x$.  In this notation, the outer vertices of $H'_i$ are adjacent to the inner vertices of $H'_{i+1}$ through the spokes of type $c$ -- we recall that the graph $G$ 
is connected. This means that we can connect the copy of $u_x$ in $H'_i$ to the copy of $v_x$ in $H'_{i+1}$
using spokes of type $c$.

We  describe the $2$-hooked construction in the case where $H'_0$ has a hamilton path $\hat P$ with origin and terminus in $u_0$ and $u_b$; consequently, each graph $H'_i$ contains the copy of the path $\hat P$. The hamilton path $\hat P$ in $H'_0$ has at least one inner edge, say $v_{-x}\,v_{-x+b}$, since it has the same number of outer and inner vertices. 
Without loss of generality  we can assume that $v_{-x}$ precedes $v_{-x+b}$ in  $\hat P$, so that the removal of the edge $v_{-x}\,v_{-x+b}$ yields the two subpaths
$u_0\,\hat P\,v_{-x}$ and $u_b\,\hat P\,v_{-x+b}$.  We can also find a hamilton path $v_0\,P\,v_b$ from $v_0$ to $v_b$ in $H'_0$, since we are assuming that $H'_0$
has a hamilton cycle, say $\hat C$, containing outer and inner edges.  A hamilton path $v_0\,P\,v_b$ in $H'_0$ can be obtained by removing an arbitrary edge of type $b$,  say  $v_yv_{y+b}$, from $\hat C$,  and by adding $m-y$ modulo $m$ to the subscripts of the vertices of $G$.

We now show the $2$-hooked construction. If $G'$ has two components,  namely $H'_0$ and $H'_1$, then  find a hamilton cycle in $G$ by connecting 
the hamilton path $u_0\,\hat P\,u_b$ in $H'_0$ to the copy of the hamilton path $v_0\,P\,v_b$ in $H'_1$ through the edges $u^0_0\,v^1_0$ and $u^0_b\,v^1_{b}$.

Now consider the case where $G'$ has more than two components,  i.e., $\gamma >1$. In $H'_i$ with $1\leq i\leq \gamma-1$,  we consider the subpaths  $u^i_0\,\hat P\,v^i_{-x}$ and $u^i_b\,\hat P\,v^i_{-x+b}$  corresponding to the subpaths $u_0\,\hat P\,v_{-x}$ and $u_b\,\hat P\,v_{-x+b}$ of $H'_0$.  We turn the subpaths $u^i_0\,\hat P\,v^i_{-x}$ and $u^i_b\,\hat P\,v^i_{-x+b}$ into the subpaths $u^i_{ix}\,\hat P\,v^i_{(i-1)\,x}$ and $u^i_{b+ix}\,\hat P\,  v^i_{b+(i-1)\,x}$ by adding $ix$ modulo $m$ to the subscripts of the vertices in $H'_i$.

For $1\leq i\leq \gamma-2$,  we join the path $u^i_{ix}\,\hat P\,v^i_{(i-1)\,x}$  in $H'_i$  to the path $v^{i+1}_{i\,x}\,\hat P\,u^{i+1}_{(i+1)x}$ in $H'_{i+1}$  by the spoke $u^i_{i\,x}\,v^{i+1}_{i\,x}$, and
also join the path $u^i_{b+ix}\,\hat P\,v^i_{b+(i-1)\,x}$ in $H'_i$ to the path $v^{i+1}_{b+i\,x}\,\hat P\,u^{i+1}_{b+(i+1)x}$ in $H'_{i+1}$ by the spoke $u^i_{b+i\,x}\,v^{i+1}_{b+i\,x}$.
We obtain two vertex disjoint paths -- the former from $v^1_0$ to $u^{\gamma-1}_{(\gamma-1)x}$ and the latter from $v^1_{b}$ to $u^{\gamma-1}_{b+(\gamma-1)x}$ -- 
whose union cover all the vertices in $G-(H'_0\cup H'_{\gamma})$.  We connect the two paths to the hamilton paths $u^0_0\,P\,u^0_b$ in $H'_0$ and  
$v^{\gamma}_{(\gamma-1)x}\,P\,v^{\gamma}_{b+(\gamma-1)x}$ in $H'_{\gamma}$ by adding the spokes $u^0_0\,v^1_0$, $u^0_b\,v^1_{b}$ and 
$u^{\gamma-1}_{(\gamma-1)x}\,v^{\gamma}_{(\gamma-1)x}$, $u^{\gamma-1}_{b+(\gamma-1)x}\,v^{\gamma}_{b+(\gamma-1)x}$. 
This gives a hamilton cycle in the graph $G$.
\end{proof}

Now we are ready to prove Lemma \ref{lem_ab_coprime2} that will be crucial in the proofs of  our main results, which are 
established in Section \ref{sec:main}.

\begin{lemma}\label{lem_ab_coprime2}
Let $G=B(m; a, S,b)$ be a  bicirculant graph  with $m>5$,  $|S|\ge 3$,  $\gcd(m,S)>1$, $a,b \ne m/2$ both   coprime to $\gcd(m,S)$ and  $b\not\equiv\pm a\pmod{\gcd(m,S)}$.
Assume that $G$ has a $(\lambda, \mu,\rho)$-non-uniform representation with $\rho=0$ and $\mu>1$. 

Further to above, assume that  for some $c\in S \smallsetminus\{0\}$ the connected components of both   subgraphs  $H(m; S\smallsetminus\{c\})$ and $B(m; a, S\smallsetminus\{c\}, b)$ are hamiltonian. Then the graph $G$ is hamiltonian.
\end{lemma}

\begin{proof}
By the assumptions, the graph $G$ has a subgraph $G'=B(m; a, S\smallsetminus\{c\}, b)$ whose connected components, possibly one, are hamiltonian. If $G'$ is connected, then the assertion is straightforward. 
Let us consider the case where $G'$ is disconnected. Notice that $H(m; S\smallsetminus\{c\})$ is disconnected since $\gcd(m, S)>1$. We set $S'=S\smallsetminus\{c\}$.

By Proposition \ref{pro:d_connected_components}, every component of $G'$  is isomorphic to the bicirculant graph $B(m;a/\delta,S'/\delta,b/\delta)$, where $\delta=\gcd(m,a,S',b)$. Moreover,  by Lemma \ref{lem_rho=02}, every component of $G'$ has a $(\lambda,\mu,\rho)$-non-uniform representation with $\rho=0$,  
and $\mu>1$  and the parameters
$a/\delta$, $b/\delta$ are coprime to $\gcd(m,S')/\delta$.  
Since $\rho=0$ and the assumptions in Lemma \ref{lem_extension_method1} are satisfied, each component of $G'$ has a hamilton path from the copy of $u_0$ to the copy of $u_b$.  
Since each component of $G’$ contains more than one component of $H(m; S’)$, every hamilton cycle in the components of $G’$ necessarily contains inner and outer edges.
Therefore, we can apply the $2$-hooked construction from Proposition \ref{rem_2hooked2} 
and find a hamilton cycle in $G$. The assertion follows.
 \end{proof}

We report the definition of a brick product and a result on hamiltonicity related to brick products from \cite{AlZh1989}, which we will use in the 
proof of Lemma \ref{lem_ab_coprime3}, which is the second main result of this section.

\begin{definition}\label{def_brick_product2}
The brick product of $C_n$ with $P_k$,  $k\ge 1$,  is denoted $C^{[k+1]}_n$ and is
defined as follows.  Let $V (C_n) =\{w_1, w_2,..., w_n\}$,  $E(C_n) =\{w_1w_2,..., w_iw_{i+1},..., w_1w_n\}$, 
$V (P_k) = \{z_1, z_2,..., z_{k+1}\}$,  and $E(P_k) = \{z_1z_2,\ldots, z_iz_{i+1},\dots, z_kz_{k+1}\}$. 
The vertex-set of $C^{[k+1]}_n$ is the cartesian product $V(C_n)\times V(P_k)$. 
The edge-set consists of all pairs of the form $(w_i, z_t)(w_{i+1}, z_t)$ and $(w_1, z_t) (w_n, z_t)$, where
$i = 1,2,\ldots, n-1$ and $t = 1,2,\ldots, k+1$,  together with all pairs of the form $(w_i, z_t) (w_i, z_{t+1})$, where $i+t\equiv 0\pmod 2$,  $i = 1 , 2 , \ldots, n$ 
and $t =1,2,\ldots, k$. 
\end{definition}

\begin{lemma}[\protect{\cite[Corollary 2.3]{AlZh1989}}] \label{lem_brick2}
Consider the brick product $C^{[k]}_n$ with $n$ even.  Let $C_{n,1}$ and $C_{n,k}$ denote the two $n$-cycles in $C^{[k]}_n$ on the vertex-sets $\{(w_i, z_1) : i = 1,2,\ldots,n\}$ and $\{(w_i, z_k) : i = 1,2,\ldots,n\}$, respectively.  Let $F$ denote an arbitrary perfect matching joining the vertices of degree $2$ in $C_{n,1}$ with the vertices of degree $2$ in $C_{n,k}$.  If $X$ is the graph obtained by adding the edges of $F$ to $C^{[k]}_n$, then $X$ has a hamilton cycle.
\end{lemma}

\begin{lemma}\label{lem_ab_coprime3}
Let $G=B(m; a, S,b)$ be a  bicirculant graph  with $m>5$,  $|S|\ge 3$,  $\gcd(m,S)>1$, $a,b \ne m/2$ both   coprime to $\gcd(m,S)$ and $b\equiv\pm a\pmod{\gcd(m,S)}$.  Assume that the connected components of $H(m; S)$ are hamiltonian. Then  $G$ is hamiltonian. 
\end{lemma}

\begin{proof}
Let $H_0,..., H_g$ be the connected components of $H(m;S)$, where $g=\gcd(m,S)-1$; $H_0$ is the component containing the vertex $u_0$. In this notation, the outer and inner vertices of $H_i$ are adjacent to the outer and inner vertices of $H_{i+1}$ through the edges of type $a$ and $b$ for $0 \le i \le g$, since $b\equiv a\pmod{g+1}$. More precisely, the copy of a vertex $u_x$ from $H_0$ in $H_i$ is $u_{x+ai}$ and we denote it by $u_x^i$; the copy of a vertex $v_x$ from $H_0$ in $H_i$ is $v_{x+ai}$ and we denote it by $v_x^i$. So   we can join the copy of the vertex $u_x$ in $H_i$ to
the copy of the vertex $u_{x}$ in $H_{i+1}$.  As for  the inner vertices, the copy of the vertex $v_x$ in $H_i$ is adjacent to the copy of the vertex $v_{x-a+b}$
in $H_{i+1}$ since $v_x^i=v_{x+ai}$ is adjacent to $v_{x+ai+b}=v_{x+a(i+1)-a+b}=v_{x-a+b}^{i+1}$. We consider cases where $g+1$ is odd and where $g+1$ is even separately.
\medskip


\noindent
\textbf{Case 1: $g+1$ is odd.}
\medskip

\noindent
Let $C$ be a hamilton cycle in $H_0$.  For convenience of notation, we will represent $C$ as the union of the hamilton path $v_s\, P\, u_0$, from $v_s$ to $u_0$,
and the edge $u_0v_s$,  i.e., $C=v_s\, P\, u_0\cup u_0v_s$. We will use the notation introduced in Section \ref{sec_extension_method1} for the vertices in $v_s\, P\, u_0$.  Once again, we point out that the construction we are going to define is independent of the edges of type $0$ that appear in the notation of $v_s\, P\, u_0$.  We treat separately the case where  $a=b$ or $a=-b$, and the case where $a\ne\pm b$.  
In both cases we give a list of paths that will be joined to form a hamilton cycle in $G$.

\paragraph{Case 1.1: $a=b$  or $a=-b$.} 
Without loss of generality we may assume that $a=b$, since otherwise we can replace  $b$ with $-b$.
We consider the path $v_s\,P\,u_0$  in $H_0$,  the  isolated vertex $u^i_0$ together with a copy of the path $v_s\,P\,v_0$  in each $H_i$ with $1\le i\le g-1$,
and a copy of $u_0\,C\,v_0=C-u_0v_0$ in $H_{g}$.

We connect the above paths and form a hamilton cycle in $G$ through the following edges:
the edges from the copy of $v_s$ in $H_i$ to the copy of $v_s$ in $H_{i+1}$ for every $0\le i\le g-1$, $i$ even;
the edges from the copy of $v_0$ in $H_i$ to the copy of $v_0$ in $H_{i+1}$ for every $1\le i\le g-1$, $i$ odd;
the edges from the copy of $u_0$ in $H_i$ to the copy of $u_0$ in $H_{i+1}$ for every $0\le i\le g-1$. The assertion follows if $a=b$ or $a=-b$.

\paragraph{Case 1.2: $a\ne\pm b$.}
Without loss of generality we can assume that $b\equiv a\pmod{g+1}$,  with $b\ne\pm a$. 
In our notation,  the outer vertices of $H_i$ are adjacent to the outer vertices of $H_{i+1}$, so that we can join the copy of the vertex $u_x$ in $H_i$ to
the copy of the vertex $u_{x}$ in $H_{i+1}$.  As for  the inner vertices, the copy of the vertex $v_0$ in $H_i$ is adjacent to the copy of the vertex $v_{-a+b}$
in $H_{i+1}$. This is due to the fact that $v_b$ is in $H_1$, since $b\equiv a\pmod{g+1}$,  with $b\ne\pm a$, and $H_1$ can be obtained from $H_0$ by adding $a$ modulo $m$
to the subscripts of the vertices in $H_0$; so $v_{0+ai}=v_0^i$ is adjacent to $v_{0+ai+b}=v_{0+a(i+1)-a+b}=v_{0-a+b}^{i+1}$.

We denote with $u_p$, $u_q$ the vertices of $v_s\,P\,u_0$ that are adjacent to $v_{-a+b}$, and assume that they occur in the order $v_s$, $u_p$, $v_{-a+b}$,  $u_q$.
Let $u_0\,P'\,u_p$ and  $v_{-a+b}\,P'\,v_0$ be the paths we obtain by deleting the edges $u_0v_0$ and $u_pv_{-a+b}$ from $C$, and let $u_p\,P'\,v_{-a+b}$ be the path we obtain by removing the edge $u_pv_{-a+b}$ from $C$.

We consider a copy of $C-u_0v_0$ in $H_0$, a copy of the paths $u_0\,P'\,u_p$ and   $v_{-a+b}\,P'\,v_0$ in each $H_i$ with $1\le i\le g-1$,  and a copy of $u_p\,P'\,v_{-a+b}$ in $H_g$.
We connect the paths, and form a hamilton cycle in $G$, by adding the following edges: the edges from the copy of $u_0$  in $H_i$ to the copy of $u_0$ in $H_{i+1}$
for every $0\le i\le g-2$, $i$ even,  the edges from the copy of $u_p$ in $H_i$ to the copy of $u_p$ in $H_{i+1}$
for every $1\le i\le g-1$, $i$ odd,  the edges from the copy of $v_0$ in $H_i$ to the copy of $v_{-a+b}$ in $H_{i+1}$ for every $0\le i\le g-1$. The assertion follows for $g+1$ odd and $a\ne\pm b$.
\medskip


\noindent
\textbf{Case 2: $g+1$ is even.}
\medskip

\noindent
Let $C$ be a hamilton cycle in the component $H_0$ and denote by $C_i$ the copy of $C$ in the component $H_i$ for $i=0,\dots,g$. We connect the cycles $C_0,\dots,C_g$ through the following edges of type $a$ and $b$:
\begin{itemize}
\item the edges from the copy of $u_x$ in $H_i$ to the copy of $u_x$ in $H_{i+1}$ for every $0\le i\le g-1$, $i$ even, and every $0 \le x < m$;
\item the edges from the copy of $v_x$ in $H_i$ to the copy of $v_{x-a+b}$ in $H_{i+1}$ for every $1\le i\le g-1$, $i$ odd, and every $0 \le x < m$.
\end{itemize}
Observe that the subgraph $X$ of the graph $G$ obtained in this way is isomorphic to the brick product of the cycle $C_{2m}$ and the path $P_{g+1}$. 
The edges from the copy of $v_x$ in $H_g$ to the copy of $v_{x-a+b}$ in $H_{0}$ for every every $0 \le x < m$ form a perfect matching $F$ joining the vertices of degree two in the first and the last cycle of $X$. By Lemma \ref{lem_brick2}, the graph $X \cup F$ has a hamilton cycle. It follows that the graph $G$ is hamiltonian.
\end{proof}

\section{Proofs of the main theorems}\label{sec:main}

This section contains the proofs of the main theorems of this paper, namely Theorems \ref{th_bicirculant_s3} -- 
\ref{th_cyclich_Haar_then_bicirculant_ham} and Theorem \ref{thm:primes_3rd_paper}, together with some auxiliary results that are suggested from the proof of Theorem \ref{th_bicirculant_s3} and are used to prove Theorems \ref{cor_main} and  \ref{th_cyclich_Haar_then_bicirculant_ham}.
We denote by $\B(m;d,s)$  the class of $d$-valent bicirculants of order $2m$ whose vertices are incident to $s$ spokes. 

\medskip

\noindent
{\bf Proof of Theorem \ref{th_bicirculant_s3}.} See Theorem \ref{th_bicirculant_s2} for $s\le 2$. 
In the following we prove that a connected bicirculant graph $G=B(m;R, S, T)$ with $|S|=s=3$ is hamiltonian, by addressing various special cases separately.

The assertion is true for $m\le 5$ by Lemma \ref{lemma_bicirculant_m<5}. It is also true for $m>5$ in the case where $H(m; S)$ is connected by Theorem \ref{th_cyclic_Haar_s3}. In particular, the assertion is  true if $R$ and $T$ are empty, that is, $G$ is a cyclic Haar graph.
Therefore, in the following we assume that $m>5$, $\gcd(m, S)>1$ and $R$, $T$ are non-empty. 

If $R\cup T$ contains at least one element that is not coprime to $\gcd(m, S)$, then the existence of a hamilton cycle in $G$ is a consequence of Theorem \ref{th_main}. In particular, it is a consequence of Theorem \ref{th_m/2} in the case where $m/2\in R, T$.

From now on assume that the sets $R,T$ are non-empty, they do not contain  $m/2$,
and every element from $R \cup T$ is coprime to $\gcd(m,S)$. If there exist elements $a\in R$ and  $b\in T$ such that $b\equiv\pm a\pmod{\gcd(m, S)}$, or such that the subgraph $B(m; a, S,b)$ has a $(\lambda,\mu,\rho)$-non-uniform representation with $\rho>0$ and $\mu>1$, then we can find a hamilton cycle in $B(m; a, S,b)$ using 
Lemma \ref{lem_ab_coprime3} and Lemma  \ref{lem_extension_method1} (ii), respectively.
A hamilton cycle in $B(m; a, S,b)$ is also a hamilton cycle in $G$, since $B(m; a, S,b)$ is a connected subgraph of $G$ spanning its vertices.  

It remains to consider the case where for every pair $(a, b)$, with $a\in R$ and $b\in T$, we have $b \not\equiv\pm a\pmod{\gcd(m,S)}$  and the subgraph $B(m; a, S,b)$ does not have a $(\lambda,\mu,\rho)$-non-uniform representation with $\rho>0$ and $\mu>1$.  Pick arbitrary $a \in R$ and $b \in T$. Then the graph $B(m; a, S,b)$ has a $(\lambda,\mu,\rho)$-non-uniform representation with $\rho=0$ and $\mu>1$ by Lemma \ref{lem_nonuniform_representation}. We show that the assumptions of Lemma \ref{lem_ab_coprime2}
hold for the graph  $B(m; a, S,b)$.
In fact, let us remove an arbitrary nonzero element $c$ from $S$ and denote with $S'$ the set of the remaining elements of $S$.  
The connected components of $B(m; a, S',b)$, possibly one,  are hamiltonian by Theorem \ref{th_bicirculant_s2}, since they are generalized rose window graphs. 
The connected components of $H(m; S')$ are also hamiltonian as they are cycles since $|S'|=2$. The existence of a hamilton cycle in $G$ now follows from  Lemma \ref{lem_ab_coprime2}.
It is thus proved that every connected graph from the class $\mathcal B(m; d, 3)$  is hamiltonian, for every $d\ge 3$.
\qed

\medskip

The proof of Theorem \ref{th_bicirculant_s3} suggests the following partition of the class of connected bicirculant graphs with $\gcd(m, S)>1$ that are not covered by Theorems \ref{th_m/2} and \ref{th_main}, and are not
cyclic Haar graphs: bicirculants of type I and type II.

\begin{definition}
Let $G=B(m; R, S, T)$ be a connected bicirculant graph with $|S|\ge 4$, $\gcd(m, S)>1$, $m/2\not\in R\cup T$, $R\cup T\ne\emptyset$, and all elements in $R\cup T$ coprime to $\gcd(m, S)$. We say that $G$ is of 
\begin{itemize}
\item \emph{type I} if there exist $a\in R$, $b\in T$ such that $b\equiv\pm a\pmod{\gcd(m, S)}$, or such that the subgraph $B(m; a, S,b)$ has a $(\lambda,\mu,\rho)$-non-uniform representation with $\rho>0$ and $\mu > 1$;
\item \emph{type II} if for every pair $(a, b)\in R\times T$  we have $b\not\equiv\pm a\pmod{\gcd(m,S)}$  and
the subgraph $B(m; a, S,b)$ does not have a $(\lambda,\mu,\rho)$-non-uniform representation with $\rho>0$ and $\mu > 1$; i.e., it has a $(\lambda,\mu,\rho)$-non-uniform representation with $\rho=0$ and $\mu > 1$.
\end{itemize}
\end{definition}

The following result holds for bicirculant graphs of type I and II. 

\begin{proposition}\label{pro_main_3rdpaper}
Let $G=B(m; R, S, T)$ be a connected bicirculant graph  with $m>5, $ 
$|S|\ge 4$, $\gcd(m, S)>1$, $m/2\not\in R\cup T$, $R\cup T\ne\emptyset$, and all elements in $R\cup T$ coprime to $\gcd(m, S)$. Then the following statements hold.
\begin{itemize}
\item[(i)] If $G$ is of type I and the connected components of $H(m; S)$ are hamiltonian, then $G$ is hamiltonian.
\item[(ii)] If $G$ is of type II and $|S|=4$, then $G$ is hamiltonian.
\end{itemize}
\end{proposition}

\begin{proof}
(i) Assume that the graph $G$ is  of type I and the connected components of $H(m; S)$ are hamiltonian.  
    Since $G$ is type I, there exist elements $a\in R$, $b\in T$ such that $b\equiv\pm a\pmod{\gcd(m, S)}$, or such that the subgraph $B(m; a, S,b)$ has a $(\lambda,\mu,\rho)$-non-uniform representation with $\rho>0$ and $\mu>1$. In the first case  we can find a hamilton cycle in $B(m; a, S,b)$ using  Lemma  \ref{lem_ab_coprime3}. In the second case we can find a hamilton cycle in $B(m; a, S,b)$ using  Lemma \ref{lem_extension_method1} (ii). A hamilton cycle in $B(m; a, S,b)$ is also a hamilton cycle in $G$, since $B(m; a, S,b)$ is a connected subgraph of $G$ spanning its vertices. Thus, it  is proved that a graph $G$ of type I is hamiltonian if the connected components of $H(m; S)$ are hamiltonian.
\medskip

    \noindent
    (ii) Let $G$ be of type II and let $|S|=4$. Then there exist elements $a\in R$, $b\in T$ such that $b \not \equiv\pm a\pmod{\gcd(m, S)}$ and the subgraph $B(m; a, S,b)$ has a  $(\lambda,\mu,\rho)$-non-uniform representation with $\rho=0$ and $\mu>1$. 
    We show that the assumptions of Lemma \ref{lem_ab_coprime2} hold.

We remove an arbitrary nonzero element $c$ from $S$ and denote with $S'$ the set of remaining elements of $S$.  
The connected components of $B(m; a, S',b)$, possibly one,  are hamiltonian by Theorem \ref{th_bicirculant_s3}, since they are graphs belonging to the class $\mathcal B(m; d, 3)$. 
The connected components of $H(m; S')$ are hamiltonian by Theorem \ref{th_cyclic_Haar_s3} since $|S'|=3$. 
The existence of a hamilton cycle in $G$ then follows from Lemma \ref{lem_ab_coprime2}.
The assertion follows.
\end{proof}

The combination of Theorem \ref{th_generalized_dihedral} and Proposition \ref{pro_main_3rdpaper} 
allows us to prove Theorems \ref{cor_main} and \ref{th_cyclich_Haar_then_bicirculant_ham}.
\medskip

\noindent
{\bf Proof of Theorem \ref{cor_main}.} If $\gcd(m,S)=1$, then $G$ is hamiltonian since $H(m; S)$ has a hamilton cycle by Theorem \ref{th_generalized_dihedral}. 
In particular, this is true in the case where $G$ is a connected cyclic Haar graph.
The statement is true even when $m/2\in R, T$, since Theorem \ref{th_m/2} holds.
In the following we consider the case where $\gcd(m, S)>1$, $m/2\not\in R\cup T$, and $G$ is not a cyclic Haar graph, that is,
$|R|=|T|\ge 2$.

Since $m/\gcd(m,S)$ is even, the connected components of $H(m; S)$ are hamiltonian by Theorem \ref{th_generalized_dihedral}. 
Consequently, the graph $G$ is hamiltonian if $R\cup T$ contains at least one element that is not coprime to $\gcd(m, S)$, since
Theorem \ref{th_main} holds. It remains to study the case where
$G$ is not covered by Theorems \ref{th_m/2} and \ref{th_main}, that is, $G$ is a graph of type I or II. 

If $G$ is of type I, then $G$ is hamiltonian since the assumptions in Proposition \ref{pro_main_3rdpaper}(i) are satisfied.  If $G$ is of type II and $|S|=4$, then the assertion follows from Proposition \ref{pro_main_3rdpaper}(ii).

We consider the case where $G$ is of type II and $|S|\ge 5$. Since $m/\gcd(m,S)$ is even, there exists $c'\in S$ that is not divisible by the greatest power of $2$ dividing $m$. Then $m/\gcd(m, S')$ is even for every proper subset $S'$ of $S$ containing $c'$. This means that the connected components of $H(m; S')$ are hamiltonian  by Theorem \ref{th_generalized_dihedral}.  

In the rest of the proof, $S'$ will denote a subset of $S$ containing $c'$, that is, $m/\gcd(m, S')$ is even. The element $c'$ will never be removed from $S$ in the iterative process that we are going to define.

We prove the claim for $|S|=5$. We consider a subgraph $G'=B(m; a, S, b)$ of $G$ with $a, b\ne m/2$. The graph $G'$ is a connected spanning subgraph of $G$, since $G$ is of type II. We remove an element $c$ from $S$ that is distinct from $0,c'$,
and denote with $S'$ the set of remaining elements of $S$.  
The connected components of $B(m; a, S', b)$ are bicirculant graphs $B(m/\delta; a/\delta, S'/\delta, b/\delta)$ with $\delta=\gcd(m, a, S', b)$.
They are still graphs of type II, since  Lemma \ref{lem_rho=02} 
holds, and belong to the class $\mathcal B(m; d, 4)$ with $d>4$. 
They are hamiltonian by Proposition \ref{pro_main_3rdpaper}(ii). Therefore, if $B(m; a, S', b)$ is connected, then $G$ is hamiltonian. 

We now deal with the case where $B(m; a, S', b)$ is disconnected. By the previous remark, the components of $B(m; a, S', b)$ are hamiltonian.
Similarly, the components of $H(m; S')$ are hamiltonian by Theorem \ref{th_generalized_dihedral}. 
Then $G'$, and consequently $G$, is hamiltonian since the assumptions in Lemma \ref{lem_ab_coprime2} 
are satisfied. It is thus proved that every 
connected bicirculant graph $G=B(m; R, S, T)$ of type II, with $|S|=5$ and $m/\gcd(m, S)$ even, is hamiltonian.

In a similar way we now we prove that $G=B(m; R, S, T)$ of type II is hamiltonian if $|S|=6$.
We consider a subgraph $G'=B(m; a, S, b)$ of $G$ with $a, b\ne m/2$. The graph $G'$ is a connected spanning subgraph of $G$, since $G$ is of type II. We remove an element 
$c$ from $S$ that is distinct from $0,c'$,
and denote with $S'$ the set of remaining elements of $S$.  
The connected components of $B(m; a, S', b)$ are bicirculant graphs $B(m/\delta; a/\delta, S'/\delta, b/\delta)$
with $\delta=\gcd(m, a, S', b)$.  They are also graphs of type II by Lemma \ref{lem_rho=02},
and belong to the class $\mathcal B(m; d, 5)$ with $d>5$.
Since they have order $2m/\delta$ and $(m/\delta)/gcd(m/\delta, S'/\delta)=m/\gcd(m, S')$ is even,
they are hamiltonian because of the previous result on bicirculant graphs of type II 
belonging to the class $\mathcal B(m; d, 5)$ that have $m/\gcd(m, S)$ even. 
Therefore, if $B(m; a, S', b)$ is connected, then $G$ is hamiltonian. For the case where $B(m; a, S', b)$ is disconnected, we can repeat the same arguments as in the case with $|S|=5$, and find a hamilton cycle in the graph $G'$, and consequently in $G$, from Lemma \ref{lem_ab_coprime2}.
It is thus proved that every graph $G=B(m; R, S, T)$ of type II, with $|S|=6$ and $m/\gcd(m,S)$ even, is hamiltonian.

We can iterate the process and repeat the above arguments for graphs $G=B(m; R, S, T)$ of type II with $|S|\ge 7$.  Thus, we have that every graph $G=B(m; R, S, T)$ of type II, with $|S|\ge 7$ and $m/\gcd(m,S)$ even, is hamiltonian. The assertion follows.
\qed

\medskip

\noindent
{\bf Proof of Theorem \ref{th_cyclich_Haar_then_bicirculant_ham}.} 
Let $G=B(m;R,S,T)$ be a connected bicirculant graph with $|S|=4$ and assume that every connected cyclic Haar graph of valence at least 4 is hamiltonian. We need to prove that $G$ is hamiltonian.
The statement is straightforward when $\gcd(m, S)=1$, since in this case $H(m; S)$ is a connected spanning subgraph of
$G=B(m; R, S, T)$ that has a hamilton cycle. The assertions follows from Theorems \ref{th_m/2} or \ref{th_main}
in the case where $\gcd(m, S)>1$ and $m/2\in R, T$, or $R\cup T$ contains at least one
element that is not coprime to $\gcd(m, S)$. In the following we consider $\gcd(m, S)>1$ and also assume that $G$ is not covered by 
Theorems \ref{th_m/2} or \ref{th_main}, that is, $G$ is a graph of type I or II.

The existence of a hamilton cycle in $G$ is a consequence of Proposition \ref{pro_main_3rdpaper}(i) if 
$G$ is a graph of type I. It is a consequence of Proposition \ref{pro_main_3rdpaper}(ii) if $G$ is a graph of type II with $|S|=4$.
It remains to study the case where $G$ is a graph of type II with $|S|\ge 5$.

We first prove the claim for $|S|=5$. We consider a subgraph $G'=B(m; a, S, b)$ of $G$ with $a, b\ne m/2$. The graph $G'$ is a connected
spanning subgraph of $G$, since $G$ is of type II. We remove an element $c \ne 0$ from $S$ and denote with $S'$ the set of remaining elements of $S$.  
The connected components of $B(m; a, S', b)$ are graphs of type II, since  Lemma \ref{lem_rho=02}  holds, that belong to the class
$\mathcal B(m; d, 4)$ with $d>4$. They are hamiltonian because of the first part of the proof on graphs of type II that belong to the class
$\mathcal B(m; d, 4)$ (equivalently, they are hamiltonian because of Proposition \ref{pro_main_3rdpaper}(ii)). 
Therefore, if $B(m; a, S', b)$ is connected, then $G$ is hamiltonian. 

We consider the case where $B(m; a, S', b)$ is disconnected. By the previous remark, the components of $B(m; a, S', b)$ are hamiltonian.
By the assumptions, the connected components of $H(m; S')$ are hamiltonian. Then $G'$, and consequently $G$, is hamiltonian since the assumptions in Lemma \ref{lem_ab_coprime2}
are satisfied. It is thus proved that every graph $G=B(m; R, S, T)$ of type II, with $|S|=5$, is hamiltonian.

We can iterate the process and repeat the above arguments for graphs $G=B(m; R, S, T)$ of type II with $|S|\ge 6$. 
Thus, we have that every graph $G=B(m; R, S, T)$ of type II, with $|S|\ge 6$, is hamiltonian. The assertion follows.
\qed

\bigskip

We close the section with the proof of Theorem \ref{thm:primes_3rd_paper}. We will use the following auxiliary result.

\begin{lemma}\label{lemma:primes}
Let $G=B(m; R, S, T )$ be a connected bicirculant graph with $|S|\ge 4$.
If $\gcd(m,R,T)$ is a product of at most three prime powers, then $G$ is hamiltonian. 
\end{lemma}

\begin{proof}
Let $S=\{0, c_1,\ldots, c_{s-1}\}$, where $s=|S|\ge 4$. Since $G$ is connected, we have $\gcd(m, R, S, T)=1$. If $G$ contains a connected spanning subgraph that is isomorphic to a bicirculant graph whose vertices are incident to two or three spokes, then $G$ is hamiltonian by Theorem \ref{th_bicirculant_s3}. Therefore, we assume that this is not the case and we will show that then $\gcd(m, R, T)$ needs to be divisible by at least four distinct primes.

By the assumptions, we have $\gcd(m, R, T, c)>1$ for every element  $c\in S$ and $\gcd(m, R, T, c_i,c_j)>1$ for every pair of elements $c_i,c_j\in S$; see Proposition \ref{prop:connected}. Let $p$ be a prime that divides $\gcd(m, R, T, c_1)$. Since $\gcd(m, R, S, T)=1$, there exists an element $c_i\in S$ that is not divisible by $p$. Therefore, there exists a prime $q\ne p$ that divides $\gcd(m, R, T, c_1,c_i)$. Again, since $\gcd(m, R, S, T)=1$ and also $|S| \ge 4$, there exists an element $c_j\in S \smallsetminus \{c_1,c_i\}$ that is not divisible by $q$. Therefore, there exists a prime $r\ne p,q$ that divides $\gcd(m, R, T, c_i,c_j)$. It follows that $\gcd(m, R, T)$ is divisible by at least three distinct primes.

Now we assume that $\gcd(m, R, T)$ is divisible by exactly three distinct primes, namely $p,q,r$. 
Since $\gcd(m, R, S, T)=1$, there exists an element $c_k\in S \smallsetminus \{c_i,c_j\}$ that is not divisible by $r$ (possibly $c_k=c_1$).
Now it follows from $\gcd(m, R, T, c_i,c_k)>1$ and $\gcd(m, R, T, c_j,c_k)>1$ that
$c_k$ is divisible by $p,q$ and $c_j$ is divisible by $p$. We now have elements $c_i,c_j,c_k$ from $S$ such that $c_k$ is not divisible by $r$, $c_i$ is not divisible by $p$ and $c_j$ is not divisible by $q$. On the other hand  $c_k,c_i$ are both divisible by $q$,  $c_k,c_j$ are both divisible by $p$ and $c_i,c_j$ are both divisible by $r$.  But then $c_i-c_k$ is  not divisible by any of $p,r$ and $c_j-c_k$ is not divisible by $q$. It follows that $\gcd(m,R,T, c_i-c_k,c_j-c_k)=1$. Consequently, the subgraph $B(m; R, \{c_i, c_j,c_k\}, T)$ is a connected spanning subgraph of $G$ whose vertices are incident to exactly $3$ spokes, since it is isomorphic to the graph $B(m; R, \{0, c_i-c_k, c_j-c_k\}, T)$ by Lemma \ref{lemma:iso}. A contradiction.  Therefore  $\gcd(m, R, T)$ is divisible by at least four distinct primes and the assertion follows.
\end{proof}
\smallskip

\noindent
{\bf Proof of Theorem \ref{thm:primes_3rd_paper}.}
(i) We assume that $m$ is even and a product of at most four prime powers. 
Then the graph $G$ is hamiltonian by Lemma \ref{lemma:primes}
if $m$ is a product of at most three prime powers. If $G$ contains a connected cyclic Haar graph as a spanning subgraph, and in particular
if $G$ is  a cyclic Haar graph, then $G$ is hamiltonian by Theorem \ref{th_generalized_dihedral}. 
Therefore, in the following, we assume that $m$ is a product of exactly four prime powers and $G$ does not contain a connected cyclic Haar graph as a spanning subgraph. Since in this case $\gcd(m, S)>1$, and the graph $G$ is connected, it follows that  $\gcd(m, R, T)$ is a product of at most three prime powers. The existence of a hamilton cycle in $G$ then follows once again from Lemma \ref{lemma:primes}. This proves (i).

\smallskip

\noindent
(ii) We assume that $m=2^{\ell}\,\Pi^4_{i=1}p_i$, where $\ell\in\{1, 2\}$ and each $p_i$ is an odd prime. 
If $G$ contains a connected cyclic Haar graph as a spanning subgraph, possibly coinciding with $G$, then $G$ is hamiltonian by Theorem \ref{th_generalized_dihedral}. 
If $m/2\in R, T$, then  $G$ is hamiltonian by Theorem \ref{th_m/2} . If $\gcd(m, R, T)$ is a product of at most three prime powers, then $G$ is hamiltonian by Lemma \ref{lemma:primes}. We now assume that $\gcd(m, R, T)$ is a product of at least four prime powers, $\gcd(m, S)>1$, $m/2\not\in R, T$ and $|R|=|T|\ge 2$. 
Observe that $\gcd(m, R, T)$ is even if $\ell=1$. 
In fact, if $\gcd(m, R, T)$ were odd, then $\gcd(m, R, T)=\Pi^4_{i=1}p_i=m/2$ and consequently $m/2\in R, T$ since at least one of the inequalities $a\le m/2$ or $-a\le m/2$ holds for every $a\in R\cup T$. That would yield a contradiction, since we are assuming $m/2\not\in R\cup T$. For $\ell=2$, $\gcd(m, R, T)$ could be even or odd.
If $\gcd(m, R, T)$ is even, then $\gcd(m, S)$ is odd, since $G$ is connected. Consequently, $m/\gcd(m, S)$ is even
and the existence of a hamilton cycle in $G$ follows from Theorem \ref{cor_main}. 
It remains to consider the case where $\ell=2$ and $\gcd(m, R, T)$ is odd, i.e., $\gcd(m, R, T)=\Pi^4_{i=1}p_i$.
The sets $R$, $T$ consist of the elements $\pm \Pi^{4}_{i=1}p_i$
since at least one of the inequalities $a< m/2$ or $-a< m/2$ holds for every $a\in R\cup T$. In this case $R=T$ and $G$ is hamiltonian by
Theorem \ref{th_generalized_dihedral}. This completes the proof of (ii). 
\smallskip

\noindent
(iii) The claim follows from Lemma \ref{lemma:primes}; see also \cite[Theorem 1.2]{{BoPiZi2025_2ndpaper}}.

\smallskip

\noindent
(iv) We assume that  $m$ is odd and a product of at most four prime powers, and $\gcd(m,S)>1$. 
Notice that in this case $\gcd(m, R, T)$ is a product of at most three prime powers. The existence of a hamilton cycle in $G$
then follows from Lemma \ref{lemma:primes}. 
\smallskip

\noindent
(v) The assertion follows from Lemma \ref{lemma:primes}, since $m=\Pi^4_{i=1}p_i$ implies that $\gcd(m, R, T)$ is a product of at most three prime powers. 
\medskip

We now show that every connected bicirculant graph $G=B(m;R,S,T)$ of order $2m$, with $m$ even, $m< 9\, 240$ is hamiltonian. Note that $9\,240=2^3\cdot 3\cdot 5\cdot 7\cdot 11$. 
In this case, $m$ is a product of at most five prime powers. If $m$ is a product of at most four prime powers, then the assertion follows from  (i). If $m$ is a product of five prime powers,
then $m$ is of the form $m=2^{\ell}\,\Pi^4_{i=1}p_i$, where $\ell\in\{1,2\}$ and each $p_i$ is an odd prime, or $m=2\cdot 3^2\cdot 5\cdot 7\cdot p$, where $p \in \{11,13\}$.
The assertion follows from (ii) in the first two cases. If $m=2\cdot 3^2\cdot 5\cdot 7\cdot p$, then we can use the same arguments as in the proof of (ii) when $G$ contains a connected cyclic Haar graph spanning its vertices (possibly equal to $G$), or when $m/2\in R, T$ to show that $G$ is hamiltonian.
We can  repeat the same arguments as in the proof of (ii) also when $\gcd(m, R, T)$ is a product of at most three prime powers, or $\gcd(m, R, T)$ is even and a product of at least four prime powers. For the case where $\gcd(m, S)>1$,  $\gcd(m, R, T)$ is odd and a product of four prime powers, that is, $3\cdot 5\cdot 7\cdot p$ is a divisor of $\gcd(m, R, T)$,
we can repeat the arguments of (ii) by setting $\{p_i: 1\le i\le 4\}=\{3, 5, 7, p\}$. Thus, it is proved that every connected bicirculant graph of order $2m$, with $m$ even, $m< 9\, 240$, is hamiltonian.

We now consider  the existence of a hamilton cycle in a connected bicirculant graph $G=B(m;R,S,T)$ of order $2m$, with $m$ odd.
First assume that $m$ is odd, $m< 3\,465$, and the spokes induce a connected subgraph of $G$. Note that $3\,465=3^2\cdot 5\cdot 7\cdot 11$ and since $m< 3\,465$, it follows that $m$ is a product of at most four prime powers.
If $m$ is a product of at most three prime powers, then the existence of a hamilton cycle in $G$ follows from  (iii); it follows from   (v) if $m$ is a product of four prime powers. Thus, it is proved that a connected bicirculant graph of order $2m$, with $m$ odd, $m< 3\,465$, is hamiltonian if its spokes induce a connected subgraph.
Since the product of the first five odd primes is $15\, 015$ and   (iv) holds, we conclude that there exists a hamilton cycle also in  every connected bicirculant graph of order $2m$ with odd $m< 15\, 015$
whose spokes induce a disconnected subgraph.\qed

\section{Conclusions}

In this paper, we have introduced non-uniform representations of connected bicirculant graphs  and, for those that admit such representations, developed techniques for joining hamilton cycles in the connected components of their subgraphs -- specifically, those that are cyclic Haar graphs -- into a hamilton cycle of the whole graph.
These representations and techniques complement the uniform representations and constructions that we introduced in \cite{BoPiZi2025_2ndpaper}. We think they are a powerful tool for proving Conjecture \ref{conject}.

We also point out that Theorem \ref{th_cyclich_Haar_then_bicirculant_ham} reduces the problem on the existence of a hamilton cycle in arbitrary bicirculant graphs to the problem on the existence of a hamilton cycle in cyclic Haar graphs. This fact justifies 
the remark we made in \cite{BoPiZi2025_2ndpaper} that cyclic Haar graphs can be considered `prominent' graphs.

To conclude, we observe that due to Theorem \ref{thm:primes_3rd_paper}, one of the first cases
not covered by our results is, for example, $B(9240; 2310, \{0, 70, 220, 154\}, 1155)$ if $m$ is even; it is
$B(3465; 2310, \{0, 35, 77, 275\}, 1155)$ or $B(15015; 2310, \{0, 455, 3575, 1001\}, 1155)$ if $m$ is odd and
the spokes induce a connected subgraph or not, respectively. Finding a hamilton cycle in these graphs seems to be a challenge because of their size.

\section*{Acknowledgements}
Simona Bonvicini is a member of GNSAGA of Istituto Nazionale di Alta Matematica (INdAM).
Toma\v{z} Pisanski is supported in part by the Slovenian Research Agency (research program P1-0294 and research projects J1-4351, J5-4596 and BI-HR/23-24-012).
Arjana \v Zitnik is supported in part by the Slovenian Research Agency (research program P1-0294 and research projects J1-3002 and J1-4351).
\bigskip

\noindent
\textbf{Author contributions} All authors wrote the main manuscript text and reviewed
the manuscript.\\

\noindent
\textbf{Data Availability Statement} No datasets were generated or analyzed during
the current study.

\section*{Declarations}

\noindent
\textbf{Conflict of interest} The authors declare that they have no conflict of interest.

\newpage


\begin{thebibliography}{99}
\bibitem{Al11983} 
B.\,Alspach, 
The classification of {H}amiltonian generalized {P}etersen   graphs,
\emph{J. Combin. Theory Ser. B} \textbf{34} (1983), 293--312.


\bibitem{AlChDe2010} 
B. Alspach, C.C. Chen,  M. Dean, 
Hamilton paths in {C}ayley graphs on generalized dihedral groups, 
\emph{Ars Math. Contemp. } {\bf 3} (2010), 29--47.

\bibitem{AlspachLiu} B.\,Alspach, J.\,Liu, On the Hamilton connectivity of generalized Petersen graphs, \emph{Discrete Math.} \textbf{309} (2009), 5461–5473.

\bibitem{AlZh1989}
B. Alspach, C. Q. Zhang,  Hamilton cycles in cubic {C}ayley graphs on dihedral groups, \emph{Ars Combin.} \textbf{28} (1989), 101--108.

\bibitem{Taba1}
 A. Arroyo, I. Hubard,K. Kutnar, E. O'Reilly, P. \v{S}parl,  
Classification of symmetric {T}aba\v{c}jn graphs,
\emph{Graphs Combin.} \textbf{31} (2015), 1137--1153.


\bibitem{BeGePi2023a}  L. W. Berman, G.  Gévay, T. Pisanski, The Gray graph is a unit-distance graph, 
\emph{Art Discrete Appl. Math.} \textbf{8} (2025), Article p1.03.


\bibitem{BoPi2003}  M. Boben, T. Pisanski,  Polycyclic configurations, \emph{European J. Combin.} \textbf {24} (2003), 431--457.



\bibitem{BoPi2017} S. Bonvicini,  T. Pisanski, 
A novel characterization of cubic Hamiltonian graphs via the associated quartic graphs. 
\emph{Ars Math. Contemp.}
 \textbf{12} (2017), 1--24.


\bibitem{BoPiZi2025} 
S. Bonvicini, T. Pisanski, A. Žitnik,  
All generalized rose window graphs are hamiltonian, 
\emph{Graphs Combin.} \textbf{42} (2026), Article 27.



\bibitem{BoPiZi2025_2ndpaper}
S. Bonvicini, T. Pisanski, A. Žitnik,  
On the Hamiltonian Bicirculants, 
\arxiv{2510.23420}{math.CO}, 2025.

\bibitem{CoEsPi2018} 
M. Conder, I. Est\'elyi,  T. Pisanski,  
Vertex-transitive Haar graphs that are not Cayley graphs. 
In: Conder, M., Deza, A., Weiss, A. (eds)  \emph{Discrete geometry and symmetry}, volume 234 of 
\emph{ Springer Proc. Math. Stat.} pages 61--70. Springer,  Cham, 2018. 



\bibitem{Foster}
I. Z. Bouwer, W. W. Chernoff, B. Monson, Z. Star,
The Foster Census,
Charles Babbage Research Centre, 1988.


\bibitem{GT}
J. L. Gross, T. W. Tucker,
\emph{Topological Graph Theory},
Dover books on mathematics, Dover Publications, 2001.


\bibitem{Gr2009} B. Grünbaum, Configurations of points and lines.
  In \emph{Graduate Studies in Mathematics} volume 103.
American Mathematical Society, Providence, RI, 2009.


\bibitem{HlMaPi2002}
M. Hladnik, D. Maru\v{s}i\v{c},T. Pisanski, 
Cyclic {H}aar graphs,
\emph{Discrete Math.} \textbf{244} (2002),137--152.



\bibitem{HolStr1978}
W. Holszty\'{n}ski, R. F. E. Strube, 
Paths and circuits in finite groups,
\emph{Discrete Math.} \textbf{22} (1978), 263--272.


\bibitem{JeJaPi2020}  K. Jasen\v cakov\'{a}, R. Jajcay, T. Pisanski,  A new generalization of generalized Petersen graphs, 
\emph{Art Discrete Appl. Math.} \textbf{3} (2020), Paper No. 1.04, 20.


\bibitem{Taba2}
K. Kutnar, D. Maru\v{s}i\v{c}, \v{S}. Miklavi\v{c}, R.  Stra\v{s}ek,  
Automorphisms of {T}aba\v{c}jn graphs,
\emph{Filomat} \textbf{27} (2013), 1157--1164.
 
\bibitem{MALNIC2007891}
A. Malni\v c, D. Maru\v{s}i\v{c}, P. \v Sparl,
On strongly regular bicirculants,
\emph{Eur. J. Comb.} \textbf{28} (2007), 891--900.


\bibitem{Ma1981} D. Maru\v{s}i\v{c}, On vertex symmetric digraphs,  \emph{Discrete Math.} \textbf{36} (1981), 69--81.

\bibitem{Pi2007}
T.  Pisanski, 
A classification of cubic bicirculants,
\emph{Discrete Math.} \textbf{307} (2007), 567--578.

\bibitem{PiSe2013}  T. Pisanski, B. Servatius, Configurations from a graphical viewpoint.
    \emph{Birkhäuser Advanced Texts: Basler Lehrbücher.}
    [Birkhäuser Advanced Texts: Basel Textbooks].
 Birkhäuser/Springer, New York, 2013.
    

\bibitem{wilsonRW}
S.~Wilson, Rose window graphs, 
\emph{Ars. Math. Contemp.} \textbf{1} (2008), 7--19.

\end{thebibliography}
\end{document}